\documentclass[a4paper]{amsart}
\usepackage[utf8]{inputenc}
\usepackage{amsmath}
\usepackage{amssymb}
\usepackage{verbatim}
\usepackage{geometry}
\usepackage{tikz}
\usepackage{soul}
\usepackage{esint}
\usepackage[colorinlistoftodos]{todonotes}
\usepackage{menukeys}
\usepackage[colorlinks=true,urlcolor=blue,citecolor=red,linkcolor=blue,linktocpage,pdfpagelabels,bookmarksnumbered,bookmarksopen]{hyperref} 
\usepackage[hyperpageref]{backref}

\newtheorem{thm}{Theorem}[section]
\newtheorem{cor}[thm]{Corollary}
\newtheorem{prop}[thm]{Proposition}
\newtheorem{lemma}[thm]{Lemma}
\theoremstyle{definition}

\newtheorem{rem}[thm]{Remark}
\newtheorem*{ack}{Acknowledgments}

\numberwithin{equation}{section}

\subjclass[2010]{35P15, 49J40, 47A75, 52A20}
\keywords{Sobolev embeddings, convex sets, Laplacian eigenvalues, torsional rigidity, distance function, inradius.}

\title[Sharp Makai--type inequalities]{Sharp Makai--type inequalities\\
 for  the best  Poincar\'e-Sobolev constants}
\author[Pisante]{Giovanni Pisante}
\address[G.\ Pisante]{Dipartimento di Matematica e Fisica
	\newline\indent
Università degli Studi della Campania “Luigi Vanvitelli"
	\newline\indent
	viale Lincoln 5, 81100 Caserta, Italy}
\email{giovanni.pisante@unicampania.it}

\author[Prinari]{Francesca Prinari}
\address[F. Prinari]{Dipartimento di Scienze Agrarie, Alimentari e Agro-ambientali
\newline\indent 
Universit\`a di Pisa
\newline\indent
Via del Borghetto 80, 56124 Pisa, Italy}
\email{francesca.prinari@unipi.it}

\begin{document}
	
\begin{abstract}
Given  a  bounded convex open set  $\Omega\subseteq \mathbb R^N$, we prove that the Poincar\'e-Sobolev constants $\lambda_{p,q}(\Omega)$ can be bounded from below by the $p$-power of the  ratio between the perimeter of $\Omega$ and a suitable power of  its volume, with an optimal constant which is explicitly given. This generalises an old result for  torsional rigidity due to Makai when $N=2$. The proof   relies on  new geometric optimal bounds for the Lebesgue norms of the distance function from the  boundary which are of independent interest. These results allow us to give a complete picture of the sharp inequalities for $\lambda_{p,q}(\Omega)$  in terms of suitable powers of  perimeter, inradius and  volume of $\Omega$. \end{abstract}
	
\maketitle
	
\begin{center}
	\begin{minipage}{11cm}
		\small
		\tableofcontents
	\end{minipage}
\end{center}

\section{Introduction}

The aim of this paper is to provide a sharp lower bound for the quantity
\begin{equation}\label{prodotto}
\mathcal{F}_{p,q}(\Omega)=\lambda_{p,q}(\Omega)\left(\frac{|\Omega|^{1-\frac{1}{p}+\frac{1}{q}}}{P(\Omega)}\right)^p,
\end{equation}
where $\Omega\subsetneq \mathbb{R}^N$ is a bounded convex open set and either $1\le q<p<\infty$ or $1<q=p<\infty$.
Here $P(\Omega)$ denotes the perimeter of $\Omega$ in the sense of De Giorgi,  $|\Omega|$ is the $N$-dimensional Lebesgue measure of $\Omega$ and 
$\lambda_{p,q}(\Omega)$ {is the  best constant for the  Poincar\'e-Sobolev  embedding
$
W^{1,p}_0(\Omega)\hookrightarrow L^q(\Omega),
$
that is }
\[
\lambda_{p,q}(\Omega):=\inf_{\psi  \in C^{\infty}_0(\Omega)} \left\{\int_{\Omega} |\nabla \psi|^p \, dx\, :\, \int_{\Omega} |\psi|^q \, dx=1\right\}.
\]

Here and in what follows, $L^{\frac{p\,q}{p-q}}(\Omega)$ stands for $L^\infty(\Omega)$ when $p=q$ and we will write $\lambda_{p}(\Omega)$ in place of $\lambda_{p,p}(\Omega)$.

In general, the optimization problems for the shape functionals $\mathcal F_{p,q}$ are ill-posed, even within the class of simply connected open sets. Indeed, over simply connected sets, the infimum of $\mathcal F_{p,q}$ is $0$ (cf.\ Proposition \ref{prop:infimum}) and, for $p=2$ and every $1\leq q<2$, the supremum of $\mathcal F_{2,q}$ is $+\infty$ (cf.\ Proposition \ref{prop:supremum}).  For this reason, throughout the paper, we restrict ourselves  to considering convex sets, where the corresponding extremal problems become meaningful and sharp inequalities can be proved.
\subsection{Background}
In the class of convex bounded open sets, an optimal upper bound for $\mathcal F_{p,q}$ has been provided in the planar case by Poly\'a in \cite{Po} when $p=2$ and $q\in \{1,2\}$. More precisely, he showed that
\begin{equation}\label{eq:upper boundpp}
\lambda_{2}(\Omega)\leq \left( \frac{\pi}{2}\right)^2  \left(\frac{P(\Omega)}{|\Omega|}\right)^2,
\qquad \hbox{i.e. } \ \mathcal F_{2,2}(\Omega) \leq \left( \frac{\pi}{2}\right)^2,
\end{equation}
and
\begin{equation}\label{eq:upper bound-torsionp1}
T_{2}(\Omega)\ge \frac{1}{3}\,\frac{|\Omega|^3}{(P(\Omega))^2},
\qquad \hbox{i.e. } \ \mathcal F_{2,1}(\Omega) \le 3,
\end{equation}
where $T_p(\Omega)=(\lambda_{p,1}(\Omega))^{-1}$ is the so-called {\sl $p$-torsional rigidity} of $\Omega$.

For the same values of the parameters $p$ and $q$, the question of minimizing $\mathcal F_{p,q}$ appears in the pioneering paper by Makai \cite{Makai}, where it is proven that
\begin{equation}\label{eq:makai01}
\lambda_{2}(\Omega)\geq \left(\frac{\pi}{4} \frac{P(\Omega)}{|\Omega|}\right)^{2},
\qquad \hbox{i.e. } \ \mathcal F_{2,2}(\Omega) \geq \left(\frac{\pi}{4}\right)^2,
\end{equation}
and
\begin{equation}\label{makaiT}
T_{2}(\Omega)\le \frac{2}{3}\,\frac{|\Omega|^3}{(P(\Omega))^2},
\qquad \hbox{i.e. } \ \mathcal F_{2,1}(\Omega) \geq \frac{3}{2},
\end{equation}
for every bounded convex open set $\Omega\subseteq \mathbb R^2$.

Later on, for some restricted values of the parameters $p,q$ and for general dimension $N$, the optimization problems associated with the functional $\mathcal{F}_{p,q}$ have been addressed in several papers.

Let us explicitly remark that the exponent $1-\frac1p+\frac1q$ in \eqref{prodotto} is dictated by scaling considerations.
Indeed, for $t>0$,  setting $\Omega_t:=t\Omega$, taking into account that 
\[
\lambda_{p,q}(\Omega_t)=t^{-p+N\left(1-\frac{p}{q}\right)}\,\lambda_{p,q}(\Omega),
\]
we have $\mathcal F_{p,q}(\Omega_t)=\mathcal F_{p,q}(\Omega)$ for every $t>0$.

Now, the upper bound \eqref{eq:upper boundpp} has been extended to every dimension $N\geq 2$ in \cite{JoSt}; a similar estimate has been established in the anisotropic case in \cite[Theorem 4.1]{DPG} when $1<p<\infty$.
In the general case $N\geq 2$, $1<p<\infty$ and $1\leq q\leq p$, the maximisation problem for $\mathcal F_{p,q}$ has been discussed in \cite[Main Theorem]{Bra1} where, for every bounded convex open set $\Omega\subseteq \mathbb R^N$, the following extension of \eqref{eq:upper bound-torsionp1} is provided
\begin{equation}\label{eq:upper bound}
\mathcal F_{p,q}(\Omega)\leq \left( \frac{\pi_{p,q}}{2} \right)^p,
\end{equation}
 with
\[
\pi_{p,q} := \inf_{u \in C_0^{\infty}((0,1))} \left\{ \|u'\|_{L^p((0,1))}\, :\, \|u\|_{L^q((0,1))}=1 \right\}.
\]
Moreover, equality in \eqref{eq:upper bound} is never attained within  the class of bounded convex open sets, but the bound is sharp for $q\leq p$, namely
\begin{equation}\label{eq:upper bound-02}
\sup \big\{\mathcal{F}_{p,q}(\Omega): \Omega\subset \mathbb{R}^{N} \text{ bounded convex open set }\big\}
=\left( \frac{\pi_{p,q}}{2} \right)^p.
\end{equation}
The problem of minimizing  $\mathcal F_{p,q}(\Omega)$ in the class of bounded convex sets of $\mathbb R^N$ has been treated in several papers,  with the aim of generalizing the results \eqref{eq:makai01} and \eqref{makaiT} for every $p,q$ satisfying $1\leq q<p$ and $N\geq 2$ (see, e.g., \cite{Bra1, {Bra2}, {BBP}, FGL, BBPconvex, GVB}). The first inequality \eqref{eq:makai01} has been extended to every $1<p<\infty$ and every dimension $N$ in \cite[Corollary 5.1]{Bra2} where it is shown that, for every bounded convex open set $\Omega\subseteq \mathbb R^N$, the following sharp lower bound holds:
\begin{equation}\label{eq:lowerbound}
\lambda_{p}(\Omega) \ge \left( \frac{\pi_{p}}{2N}\right)^p\, \left(\frac{P(\Omega)}{|\Omega|}\right)^p,
\qquad \hbox{i.e.} \;\;\; \mathcal F_{p,p}(\Omega) \geq \left( \frac{\pi_{p}}{2N}\right)^p,
\end{equation}
where $\pi_{p}:=\pi_{p,p}$. On the contrary, the results existing in the literature and aimed at generalizing the second Makai inequality \eqref{makaiT}
are satisfactory only in the planar case. Indeed, we recall that, when $N=2$, 
in \cite[Theorem 1]{FGL} the following optimal upper bound is proved:
\begin{equation}\label{fraga}
T_{p}(\Omega)\le \left(\frac {2^{p'+1}}{(p'+1)(p'+2)} \right)^{p-1} \frac{|\Omega|^{2p-1}}{(P(\Omega))^p},
\end{equation}
i.e.
\[
\mathcal F_{p,1}(\Omega)\geq \left(\frac {(p'+1)(p'+2)} {2^{p'+1}} \right)^{p-1}, 
\]
for every bounded convex open set $\Omega\subseteq \mathbb R^2$.
Later, for the general case $N\geq 2$  and $p=2$, it has been conjectured in \cite{BBPconvex} that the sharp inequality 
\begin{equation}\label{conj}
T_{2}(\Omega)\le \frac{2N^2}{(N+1)(N+2)}\frac{|\Omega|^{3}}{(P(\Omega))^2}
\end{equation}
holds in the class of bounded convex open sets $\Omega\subseteq \mathbb R^N$. This conjecture is based on the fact that in the class of {\sl thin domains}
\[
\Omega_\varepsilon=\big\{(x,y)\in A\times \mathbb{R}\ :\ 0<y<\varepsilon h(x)\big\},
\]
where $\varepsilon$ is a small positive parameter, $A$ is a bounded convex open set of $\mathbb R^{N-1}$, and $h:A\to (0,+\infty)$ is a given Lipschitz continuous function, the inequality \eqref{conj} holds and equality  is reached asymptotically (see \cite[Theorem 4.2]{BBPconvex}).

We explicitly note that, in the planar case, the constants appearing in \eqref{conj}  and in \eqref{fraga} when $p=2$ coincide with the constant appearing in the Makai inequality \eqref{makaiT}.
\subsection{Main result}
After this preamble, we can state the main result of this paper. 
 \begin{thm}
\label{cor:upperb}
\sl Let $N\geq 2$. Let $1 \le q < p<\infty$ and let $\Omega \subset \mathbb{R}^N$ be a bounded convex open set. Then the following lower bound holds
\begin{equation}\label{eq:lower bound per}
\lambda_{p,q}(\Omega) \ge K_{p,q,N} \left(\frac{P(\Omega)}{|\Omega|^{1-\frac{1}{p}+\frac{1}{q}}}\right)^p,
\end{equation}
where
\[
K_{p,q,N}= \left( \frac{\pi_{p,q}}{2N}\right)^p\,  \prod_{j=2}^N \left(  \frac{pq}{j(p-q)}+1 \right)^{\frac{p}{q}-1}.
\]
Moreover, the equality in \eqref{eq:lower bound per} is asymptotically attained, as $\varepsilon\to 0$, by the sequence of thin domains 
{\begin{equation}\label{thinintro}
\Omega_\varepsilon=\Big\{(x,y)\in \mathbb R^{N-1}\times \mathbb{R} :\ |x|<1,\  0<y<\varepsilon (1-|x|) \Big\}.
\end{equation}}
\end{thm}
In particular we have that 
\[
\inf \big\{\mathcal{F}_{p,q}(\Omega): \Omega\subset \mathbb{R}^{N} \text{ bounded convex open set }\big\}=
K_{p,q,N},
\]
which is the counterpart of \eqref{eq:upper bound-02}. We note that, taking the limit in \eqref{eq:lower bound per}  as\footnote{we use here that $q\mapsto \lambda_{p,q}(\Omega)$ is continuous on every compact interval when $\Omega\subset\mathbb{R}^N$ is a bounded open set, see \cite[Theorem 1]{AFI} and \cite[Theorem 7]{Er}.}
  $q\nearrow p$, we find the sharp lower bound \eqref{eq:lowerbound} for $\lambda_{p}(\Omega)$. 
Moreover,  when $q=1$, as a corollary of Theorem \ref{cor:upperb},  we get a sharp upper bound for the $p$-torsional rigidity. This result both extends the inequality \eqref{fraga} to the case $N\geq 3$ and, when applied with $p=2$, shows the conjecture \eqref{conj}.
\begin{cor}\label{maincoro}
Let $N\geq 2$. Let $1 < p<\infty$ and let $\Omega \subset \mathbb{R}^N$ be a bounded convex open set. Then
\begin{equation}\label{eq:lower bound tor}
T_{p}(\Omega) \le N^p \binom{N+p'}{N}^{1-p} \frac{|\Omega|^{2p-1}}{(P(\Omega))^p}
\end{equation}
with $p'=\frac{p}{p-1}$. Moreover, the equality in \eqref{eq:lower bound tor} is  asymptotically attained, as $\varepsilon\to 0$, by the sequence of thin domains $\Omega_\varepsilon$ given by \eqref{thinintro}.
\end{cor}

The proof of Theorem \ref{cor:upperb} follows the strategy introduced by Makai to derive \eqref{makaiT}. His argument consists of relating both terms of  \eqref{makaiT} with the $L^2$-norm of the distance function from the boundary of $\Omega$, i.e.
\[
d_{\Omega}(x):=\inf\big\{|x-y|:y\in \partial\Omega\big\}.
\]
More precisely, he showed that,  for every bounded convex open set $\Omega\subseteq \mathbb R^2$, the following optimal estimates hold:
\begin{equation}\label{Torsione-dist}
T_2(\Omega)\leq \int_{\Omega} d^2_{\Omega}(x)\,dx\leq  \frac{2}{3}\,\frac{|\Omega|^3}{(P(\Omega))^2}.
\end{equation}

Now, a result by Prinari and Zagati (see \cite[Theorem 1.1]{PZ}) generalized the first inequality of  \eqref{Torsione-dist} providing a sharp lower bound for $\lambda_{p,q}(\Omega)$ in terms of a suitable power of $\|d_{\Omega}\|_{L^{\beta}(\Omega)}$, where $\beta=\frac{pq}{p-q}$. 
More precisely, they proved that, on every bounded convex open set $\Omega$ and  for every  $1 \le  q < p < \infty$, the following optimal Makai type inequality holds 		
\begin{equation}
\label{eq:prizag}
\lambda_{p,q}(\Omega) \ge \left( \frac{\pi_{p,q}}{2} \right)^p \left(\displaystyle  \frac{p-q}{pq+p-q} \right)^{\frac{p-q}{q}} {\left(\displaystyle \int_{\Omega} d_{\Omega}^{\frac{p\,q}{p-q}}\, dx \right)^{\frac{q-p}{q}}}.
\end{equation}
Inspired by the second inequality in \eqref{Torsione-dist},  we aim to   obtain some  sharp estimates, within the class of bounded convex open sets,  for the quantity 
\[
D_{\beta}(\Omega):=\int_{\Omega} d^{\beta}_{\Omega}(x)\,dx
\]
in terms of   two quantities chosen among the  perimeter  $P(\Omega)$,  the volume  $|\Omega|$ and the \emph{inradius} $r_{\Omega}$ of the set $\Omega$, where $r_{\Omega}$ is defined as the radius of the largest ball contained in $\Omega$, i.e.
\begin{equation}
\label{inradius}
r_\Omega:=\sup \big\{r>0\, :\, \mbox{ exists }x_0\in \Omega \mbox{ such that } B_r(x_0)\subset\Omega\big\}.
\end{equation}
In particular, for $\beta\geq 1$, in Theorem \ref{reverseholder} we 
prove the following sharp upper bound 
\begin{equation}\label{sharpdis}
 D_{\beta}(\Omega)
\le N^\beta \binom{N+\beta}{N}^{-1} \frac{|\Omega|^{\beta+1}}{\left(P(\Omega)\right)^\beta},
\end{equation}
which is the extension of that one shown by Makai in  \eqref{Torsione-dist} in the particular case $\beta=N=2$, and  is the key estimate of our paper. Indeed, combined with  \eqref{eq:prizag}, allows us to prove \eqref{eq:lower bound per} and \eqref{eq:lower bound tor}. 
\subsection{Plan of the paper} The paper is organized as follows.
We start, in Section \ref{intera}, by proving a sharp two-sided estimate for $D_{\beta}(\Omega)$ in terms of $r_{\Omega}^{\beta+1}P(\Omega)$ for every $\beta\geq 0$, see Theorem \ref{thm_bound}.  In Proposition \ref{controinradius-volume}, we show an optimal lower bound for $D_{\beta}(\Omega)$ in terms of $r^{\beta}_{\Omega}|\Omega|$, complementing the sharp upper bound proved  in \cite[Remark 3.1]{PZ}. Then, in Theorem \ref{reverseholder}, we prove the inequality \eqref{sharpdis} and  we show that it holds as an equality  on the bounded convex open sets which are  homothetic to their form body  (see Definition~\ref{def:homothetic}, Remark \ref{def:homothetic} and Theorem \ref{reverseholder}). The proof exploits a lower bound, shown by  Larson in \cite{lar}, for the perimeter of the inner parallel sets of a convex open set $\Omega$ by means of $P(\Omega)$ and $r_{\Omega}$ (see Formula \eqref{larson}). Moreover,  in the same theorem, we also provide a  sharp lower bound for $D_{\beta}(\Omega)$ in terms of  $P(\Omega)$ and $|\Omega|$.

\noindent 
As a joint application of \eqref{eq:prizag} and \eqref{sharpdis}, with $\beta=pq(p-q)^{-1}$, in Section \ref{proof} we provide the proof of the main results of this paper.  In  Section \ref{AppB},    we establish further geometric estimates for $\lambda_{p,q}(\Omega)$.  First of all, as a byproduct of  \eqref{eq:prizag},  in Corollary \ref{cor:inradius-perimeter} we give a sharp lower bound for $\lambda_{p,q}(\Omega)$ by means of $r_{\Omega}$ and $P(\Omega)$. Moreover,  we derive optimal upper bounds for  $\lambda_{p,q}(\Omega)$ by means of $|\Omega|$ and $r_{\Omega}$ (see Theorem \ref{teo:upperbound}) and  in terms of $P(\Omega)$ and $r_{\Omega}$  (see Corollary  \ref{cor:upper-inradius-perimeter}). The first  extends to every $1<p<\infty$  a result shown by Brasco and Mazzoleni in the particular case $p=2$ (see \cite[Theorem 1.2]{BM}). Finally,  Appendix \ref{AppA} is devoted to show  the ill-posedness of the optimization problems for $\mathcal F_{p,q}$ in the class of simply connected open sets.
\subsection{Final remarks}
As a concluding remark, we note that, by joining the new geometric estimates proved in this paper  with  some known results,    we can give a complete picture of  sharp inequalities satisfied by  $\lambda_{p,q}(\Omega)$ by means of suitable powers, depending on $p$ and $q$, of two among  the quantities   $P(\Omega)$, $|\Omega|$ and $r_{\Omega}$,  when $\Omega $ is a bounded convex open set and  $1\le q<p<\infty$ or $1<q=p<\infty$. 
\begin{list}{-}{\leftmargin=1em \itemindent=-0.3em}
\item First, thanks to  the lower bound proved in  Theorem  \ref{cor:upperb} and to  the upper bound in  \cite[Main Theorem]{Bra1},  we have the following optimal bounds for  $\lambda_{p,q}(\Omega)$  in terms  of the  perimeter and the  volume of $\Omega$
\begin{equation*}\label{per-vol}
K_{p,q,N}\leq \lambda_{p,q}(\Omega)   \left(\frac{|\Omega|^{1-\frac{1}{p}+\frac{1}{q}}}{P(\Omega)}\right)^p< \left( \frac{\pi_{p,q}}{2} \right)^p,
\end{equation*}
with the lower bound and the upper bound   asymptotically attained  by  the sequence of thin domains  $\Omega_{\varepsilon}$, defined by \eqref{thinintro},  and by the sequence  of  slab-type domains  $\Omega_L$ (see \eqref{slab}), respectively.

\item Then, by  combining Theorem \ref{teo:upperbound} with the lower bound given in  \cite[Theorem 5.7]{BPZ1}, we obtain  the following two-sided  estimate for  $\lambda_{p,q}(\Omega)$ by means of  the volume  and the inradius of $\Omega$
\begin{equation}\label{inr-vol}
 \left( \frac{\pi_{p,q}}{2} \right)^p\leq \lambda_{p,q}(\Omega)\,|\Omega|^{\frac{p-q}{q}} r_\Omega^p\le \omega_N^\frac{p-q}{q}\,\lambda_{p,q}(B_1),
\end{equation}
where  $B_1$ is the  unit open ball and  $\omega_N=|B_1|$; finally, thanks to Corollaries \ref{cor:inradius-perimeter} and \ref{cor:upper-inradius-perimeter},  we obtain the following sharp estimates for  $\lambda_{p,q}(\Omega)$  in terms of  the perimeter and the inradius of $\Omega$
\begin{equation}\label{per-inr}\left(\frac{\pi_{p,q}}{2}\right)^p
\left(\frac{1}{pq+p-q}\right)^{\frac{p-q}{q}}\leq 
 \lambda_{p,q}(\Omega)  r_{\Omega}^{\frac{pq+p-q}{q}} P(\Omega)^{\frac{p-q}{q}}\le 
 (N\omega_N)^\frac{p-q}{q}\,\lambda_{p,q}(B_1).
 \end{equation}
For both estimates  \eqref{inr-vol} and \eqref{per-inr}, the equality in the upper bound is attained  when  $\Omega=B_1$ while the lower bound is asymptotically attained by  the sequence of slab-type domains  $\Omega_L$.
\end{list}

\begin{ack} 
The authors wish to warmly thank  Lorenzo Brasco for some  useful discussions  on the subject. F.P.  thanks Anna Chiara Zagati for some preliminary conversations on this problem.
The authors are both members of the {\it Gruppo Nazionale per l'Analisi Matematica, la Probabilit\`a
e le loro Applicazioni} (GNAMPA) of the Istituto Nazionale di Alta Matematica (INdAM). They  gratefully acknowledge the financial support of the project GNAMPA 2025  ``Ottimizzazione Spettrale, Geometrica e Funzionale" ({\tt CUP E5324001950001}).  Francesca Prinari
thanks the  Dipartimento di Matematica e Fisica at
Università degli Studi della Campania “Luigi Vanvitelli” in Caserta, Italy, for its hospitality. GP is also supported by the Portuguese government through FCT - Fundação para a Ciência e a Tecnologia, I.P., within the project ASSO ({\tt DOI 10.54499/2024.14494.PEX}). 
\par

\end{ack}
\section{Sharp inequalities for the Lebesgue norms of  the distance functions}\label{intera}
\subsection{Inradius--perimeter bounds}
\label{stimeinradius}

This subsection is devoted to generalizing the sharp geometric estimate 
\begin{equation}\label{pervol}
\frac{1}{N} \le \frac{|\Omega|}{r_{\Omega} P(\Omega)} < 1,
\end{equation}
which holds for every bounded convex open set $\Omega\subseteq \mathbb R^N$ with $N\geq 2$ (see, e.g., \cite[Lemma B.1]{Bra1}).
We recall that $r_{\Omega}$, defined by  \eqref{inradius},   coincides with the supremum  of $d_\Omega$ over $\Omega$.
\begin{thm}\label{thm_bound}
For any bounded convex open set $\Omega\subseteq \mathbb R^N$ with $N\geq 2$, the following sharp two-sided inequality holds:
\begin{equation}\label{pervolext}
\frac{1}{N}\binom{N+\beta}{N}^{-1}
\leq \frac{\int_{\Omega} d^{\beta}_{\Omega}\,dx}{r_{\Omega}^{\beta+1}P(\Omega)}
<\frac{1}{\beta+1},
\qquad \forall \beta \geq 0.
\end{equation}
Moreover, the lower bound  in \eqref{pervolext} holds as equality if and only if   $\Omega$ satisfies the following condition:
\begin{equation}\label{optimal}
|\Omega|= \frac{P(\Omega) r_{\Omega}}{N},
\end{equation}
while the upper bound in \eqref{pervolext}   is asymptotically attained, as $L\to \infty$, by the  sequence of slab-type domains given by
\begin{equation}\label{slab}
\Omega_L =\left(-\frac{L}{2},\frac{L}{2}\right)^{N-1}\times (0,1).
\end{equation}
\end{thm}
Here, for $\beta\ge 0$, we understand
\[
\binom{N+\beta}{N}:=\frac{\Gamma(N+\beta+1)}{\Gamma(\beta+1)\Gamma(N+1)},
\]
where $\Gamma$ is the Gamma function.  We explicitly note that, for $\beta=0$, \eqref{pervolext} yields \eqref{pervol}.
\medskip

The proof of Theorem \ref{thm_bound} will  follow by combining Propositions \ref{lower-estim} and \ref{upper-estim}. We start by  recalling some preliminary facts.
For every $t\ge 0$, we denote by $\Omega(t)$ the {\it interior parallel set} at distance $t$ from $\partial\Omega$, i.e.
\begin{equation}\label{parallelset}
\Omega(t):=\bigl\{x\in\Omega:\ d_{\Omega}(x)>t\bigr\}.
\end{equation}
When $\Omega$ is a bounded convex open set, it is well known that $d_{\Omega}$ is concave on $\Omega$ (see \cite{AK}) and $\Omega(t)$ is convex as well.
We set $A(t):=|\Omega(t)|$ and
\[
L(t):=P(\Omega(t)).
\]
Then $A$ is a.e.\ differentiable on $(0,r_\Omega)$ and $A'(t)=-L(t)$ for a.e.\ $t\in(0,r_\Omega)$.
Moreover, the Brunn--Minkowski theorem (see \cite[Theorem 7.4.5]{S}) implies that the function $t\mapsto L(t)^{\frac{1}{N-1}}$ is concave on $[0,r_{\Omega}]$.
Finally, by \cite[Lemma 2.2.2]{BB}, $L$ is monotonically decreasing on $[0,r_{\Omega}]$.

\medskip

We start by proving the lower bound in Theorem \ref{thm_bound}.

\begin{prop}\label{lower-estim}
Let $\beta\geq 0$. Then, for every bounded convex open set $\Omega\subseteq \mathbb R^N$,  the following sharp inequality  holds 
\begin{equation}\label{lowerdistanzabeta}
\int_{\Omega} d^{\beta}_{\Omega}\,dx \geq \frac{1}{N}\binom{N+\beta}{N}^{-1}P(\Omega) r_{\Omega}^{\beta+1}.
\end{equation}
Moreover, equality  in \eqref{lowerdistanzabeta} is attained if and only if   $\Omega$ satisfies  \eqref{optimal}.

\end{prop}

\begin{proof}  Let $\Omega\subseteq \mathbb R^N$ be a  bounded convex open set.
Thanks to \cite[Theorem 1.2]{lar}, for every inner parallel set $\Omega(t)$ with $t\geq 0$, it holds
\begin{equation}\label{larson}
L(t) \geq P(\Omega)\left(1-\frac{t}{r_{\Omega}}\right)^{N-1}_+.
\end{equation}
Since $d_\Omega$ is $1$-Lipschitz and $|\nabla d_\Omega|=1$ a.e.\ in $\Omega$, the coarea formula yields
\begin{equation}\label{merco}
\int_{\Omega} d^{\beta}_{\Omega}\,dx=\int_0^{r_\Omega} t^{\beta}L(t)\,dt.
\end{equation}
Moreover,
\begin{equation}\begin{split}\label{lun2}
\int_0^{r_{\Omega}} t^{\beta}\left(1-\frac{t}{r_{\Omega}}\right)^{N-1}\,dt &=  r_{\Omega}^{\beta+1}\int_0^{1} t^{\beta}(1-t)^{N-1}\,dt=  r_{\Omega}^{\beta+1} B(\beta+1,N),
\end{split}
\end{equation}
where $B$ is the Beta function. 
Since
\[
B(\beta+1,N)=\frac{\Gamma(\beta+1)\Gamma(N)}{\Gamma(\beta+N+1)}
=\frac{1}{N}\binom{N+\beta}{N}^{-1},
\]
 by combining \eqref{merco} with \eqref{larson}, we obtain the inequality 
\begin{equation}\label{lun3}\int_{\Omega} d^{\beta}_{\Omega}\,dx
\geq P(\Omega)\int_0^{r_{\Omega}} t^{\beta}\left(1-\frac{t}{r_{\Omega}}\right)^{N-1}\,dt= \frac{1}{N}\binom{N+\beta}{N}^{-1}P(\Omega) r_{\Omega}^{\beta+1},
\end{equation}
i.e. \eqref{lowerdistanzabeta} holds.
We now note that  \eqref{optimal} is equivalent to
\begin{equation}\label{larsoncond}
L(t)=P(\Omega)\left(1-\frac{t}{r_{\Omega}}\right)^{N-1}\qquad \forall t\in (0,r_{\Omega}).
\end{equation}
Indeed, by using the coarea formula, we have that 
\begin{equation*}\begin{split} |\Omega| = \frac{P(\Omega) r_{\Omega}}{N}& \Longleftrightarrow  \int_{0}^{r_{\Omega}} L(t)\,dt= P(\Omega) \int_{0}^{r_{\Omega}} \left(1-\frac{t}{r_{\Omega}}\right)^{N-1}\,dt \\
 & \Longleftrightarrow  
\int_{0}^{r_{\Omega}}\!\!\left(L(t)-P(\Omega)\left(1-\frac{t}{r_{\Omega}}\right)^{N-1}\right)\,dt=0,\\
& \Longleftrightarrow  L(t)=P(\Omega)\left(1-\frac{t}{r_{\Omega}}\right)^{N-1}\qquad \forall t\in (0,r_{\Omega}),
\end{split}
\end{equation*}
where the last equivalence follows from \eqref{larson}.
Therefore, if a convex bounded open set  $\Omega\subseteq \mathbb R^N$ satisfies \eqref{optimal}, then, combining \eqref{merco}, \eqref{larsoncond} and \eqref{lun2}, we get that the equality in  \eqref{lowerdistanzabeta} is attained. Conversely, if  \eqref{lowerdistanzabeta} holds as an equality, then  \eqref{lun3} and \eqref{merco} imply that 
\[\int_0^{r_\Omega} t^{\beta}L(t)\,dt
= P(\Omega)\int_0^{r_{\Omega}} t^{\beta}\left(1-\frac{t}{r_{\Omega}}\right)^{N-1}\,dt,
\]
that is 
\[\int_{0}^{r_{\Omega}}\!\! t^{\beta} \left(L(t)-P(\Omega)\left(1-\frac{t}{r_{\Omega}}\right)^{N-1}\right)\,dt=0.\]
Hence, \eqref{larson} yields \eqref{larsoncond}  and this  implies the validity of   \eqref{optimal}.  \end{proof}

\begin{rem}
We note that the identity \eqref{optimal} is trivially attained by  balls. Another class of sets where it holds is given by polytopes $\Omega$ of dimension $N$ whose facets $F_k$, $k\in\{1,\dots,K\}$, are tangent to the same ball $B_R(x_0)\subseteq \Omega$. Indeed, for every $k\in\{1,\dots,K\}$ let $T_k$ be the convex hull of $F_k\cup\{x_0\}$. Without loss of generality, we can assume that $F_k\subseteq \{(x',0), x'\in \mathbb R^{N-1}\}$. Then, setting $F_k(t):=\{x'\in \mathbb R^{N-1}\,:\, (x',t)\in T_k\}$, we have that
\begin{equation*}\begin{split}
|T_k|&=\int_{0}^R \mathcal H^{N-1} (F_k(t)) dt=\int_{0}^R  \mathcal H^{N-1} \left(\frac  t R  F_k\right ) dt \\
&= \mathcal H^{N-1} \left( F_k\right )   \int_{0}^R  \left(\frac  t R  \right )^{N-1} dt=  \mathcal H^{N-1} (F_k) \frac R N,\end{split}
\end{equation*}
where  $\mathcal H^{N-1}$ is   the $(N-1)$-dimensional Hausdorff measure.
Since $r_{\Omega}=R$,   $\overline{\Omega}=\bigcup_{k=1}^K T_k$ and $|T_k\cap T_j|=0$ for $j\neq k$, we obtain
\[
|\Omega|=\sum_{k=1}^K |T_k|
=\sum_{k=1}^K \frac{\mathcal H^{N-1}(F_k)   r_{\Omega}}{N}
=\frac{P(\Omega) r_{\Omega}}{N},
\]
that is, \eqref{optimal} holds.
\end{rem}

\begin{rem}[Homothetic sets and form body]\label{def:homothetic}
In order to describe the geometric structure of sets satisfying the equality \eqref{optimal}, we recall two standard notions from convex geometry: homotheticity and the form body associated with a convex set.

Two sets $A,B\subset \mathbb{R}^N$ are said to be \emph{homothetic} if there exist $x_0\in\mathbb{R}^N$ and $\lambda>0$ such that
\[
A=x_0+\lambda B.
\]
Following \cite{S}, for a convex body $\Omega\subset\mathbb{R}^N$ we denote by $\mathcal{U}(\Omega)$ the set of outward unit normals at regular boundary points and we define its \emph{form body} by
\[
\Omega_*:=\bigcap_{u\in\mathcal{U}(\Omega)} \bigl\{x\in\mathbb{R}^N:\ \langle x,u\rangle\le 1\bigr\}.
\]
Equivalently, $\Omega_*$ is the unique convex body whose supporting hyperplanes with normals in $\mathcal{U}(\Omega)$ are tangent to the unit ball.
Now, thanks to \eqref{larsoncond} and to \cite[Theorem 1.2]{lar}, we have that $\Omega$ satisfies \eqref{optimal} if and only if $\Omega$ is homothetic to its form body. 
\end{rem}

In order to conclude the proof of Theorem \ref{thm_bound}, in the following proposition we  show the upper bound in \eqref{pervolext}. 
\begin{prop}\label{upper-estim}
Let $\beta\geq 0$. Then
\begin{equation}\label{sup}
\sup\left\{\frac{\int_{\Omega} d^{\beta}_{\Omega}\,dx}{r_{\Omega}^{\beta+1}P(\Omega)}:\, \Omega \ \hbox{bounded convex open set}\right\}
= \frac{1}{\beta+1},
\end{equation}
and the equality   is asymptotically attained, as $L\to \infty$, by the  sequence  $
\Omega_L$ given by \eqref{slab}.
\end{prop}

\begin{proof}
Let  $\Omega\subseteq \mathbb R^N$ be a bounded convex open set. Then,  thanks to \cite[Remark 3.1]{PZ}, the following sharp inequality holds:
\begin{equation}\label{prizag2}
\int_{\Omega} d_{\Omega}^{\beta}\,dx \le |\Omega|\, \frac{r_{\Omega}^{\beta}}{\beta+1},
\qquad \mbox{ for every } \beta\geq 0.
\end{equation}
Combining \eqref{prizag2} with the upper bound in \eqref{pervol}, we obtain
\[
\frac{\int_{\Omega} d^{\beta}_{\Omega}\,dx}{P(\Omega)}
\leq \frac{|\Omega|}{P(\Omega)}\, \frac{r_{\Omega}^{\beta}}{\beta+1}
< \frac{r_{\Omega}^{\beta+1}}{\beta+1}.
\]
Moreover, since \eqref{prizag2} and the upper bound in \eqref{pervol} are both asymptotically attained by the slab-type sequence $\Omega_L$ (see, e.g., \cite[Remark B.2]{Bra1}), we get \eqref{sup}.
\end{proof}

\subsection{Inradius--volume bounds}
\label{stimeinradiusvolume}

In this subsection, we focus on a counterpart of the sharp estimate \eqref{prizag2}. The following result generalizes \cite[Proposition 6.1]{BBPtorsion}, which corresponds to the case $\beta=1$.

\begin{prop}\label{controinradius-volume}
For every $\beta>0$, it holds
\begin{equation}\label{inradius-vol}
\min \left\{r^{-\beta}_\Omega \fint_{\Omega} d^{\beta}_{\Omega}\,dx :\,  \Omega \ \hbox{bounded convex open set} \right\}
= \binom{N+\beta}{N}^{-1}.
\end{equation}
The minimum is attained by $\Omega=B_1$.
\end{prop}

The proof relies on the following reverse H\"older inequality due to  Borell (we refer to \cite[Theorem 4.1]{bor73} for the case $q\ge 0$ and to \cite[Theorem 5.1]{gar98} for the extension to $q>-1$ and for the equality case).  We will use this result also in the proof of the key estimate \eqref{sharpdis}.

\begin{thm}\label{theo.borell}
Let $0\le q\le s$. Then, for every convex set $E\subseteq \mathbb R^N$ and every nonnegative concave function $f$ on $E$, we have
\[
\left(\fint_E f^s\,dx\Big)^{1/s}\le C_{s,q}\right(\fint_E f^q\,dx\Big)^{1/q},
\]
where
\[
C_{s,q}=\binom{N+q}{N}^{1/q}\binom{N+s}{N}^{-1/s}.
\]
Moreover, equality holds when $E$ is a ball of radius $1$ and $f(x)=1-|x|$.
\end{thm}

\begin{proof}[Proof of Proposition \ref{controinradius-volume}]
Let $\Omega\subset\mathbb{R}^N$ be a convex bounded open set.   Since $d_\Omega$ is concave on $\Omega$, by applying Theorem \ref{theo.borell} to $f=d_\Omega$ with $q=\beta$ we obtain
\begin{equation}\label{p-norma}
\Big(\fint_\Omega d_\Omega^s\,dx\Big)^{1/s}\le C_{s,\beta}\Big(\fint_\Omega d_\Omega^{\beta}\,dx\Big)^{1/\beta},
\qquad \forall \,s\ge \beta>0.
\end{equation}
Since \eqref{p-norma} is an equality when $\Omega=B_1$, we have
\[
\displaystyle C_{s,\beta}
=\frac{\left(\displaystyle\fint_{B_1} (1-|x|)^s\,dx\right)^{1/s}}{\left(\displaystyle\fint_{B_1} (1-|x|)^\beta\,dx\right)^{1/\beta}}
=\frac{\left(\displaystyle \fint_{B_1} (1-|x|)^s\,dx\right)^{1/s}}{\displaystyle \binom{N+\beta}{N}^{-1/\beta}}.
\]
Moreover,
\[
\Big(\fint_{B_1} (1-|x|)^s\,dx\Big)^{1/s}\to \|1-|x|\|_{L^\infty(B_1)}=1,
\qquad \text{as } s\to\infty,
\]
and therefore
\begin{equation}\label{cpbeta}
\lim_{s\to\infty} C_{s,\beta}=\binom{N+\beta}{N}^{1/\beta}.
\end{equation}
Passing to the limit as $s\to\infty$ in \eqref{p-norma} and using \eqref{cpbeta}, we obtain
\[
\|d_\Omega\|_{L^\infty(\Omega)} \le \binom{N+\beta}{N}^{1/\beta}\Big(\fint_\Omega d_\Omega^{\beta}\,dx\Big)^{1/\beta}.
\]
Since $\|d_\Omega\|_{L^\infty(\Omega)}=r_\Omega$, this yields
\[
\fint_\Omega d_\Omega^{\beta}\,dx\ge \binom{N+\beta}{N}^{-1} r_\Omega^\beta,
\]
and  the above inequality  holds as an equality when  $\Omega=B_1$, thanks to  Theorem \ref{theo.borell}.   
This concludes the proof of  \eqref{inradius-vol}.
\end{proof}

\subsection{Perimeter--volume bounds}
\label{stimeperimetrovolume}
The main result of this section is the following theorem, which is proved 
under the additional assumption $\beta\ge 1$.

\begin{thm}\label{reverseholder}
Let $\beta\ge 1$. Then, for every bounded convex open set $\Omega$, the following two-sided sharp estimate holds 
\begin{equation}\label{lowerdistanzabeta1}
\frac 1 {\beta+1}\leq \left(\frac{P(\Omega)}{|\Omega|}\right)^\beta \fint_{\Omega} d_{\Omega}^{\beta}\,dx
\le N^\beta \binom{N+\beta}{N}^{-1}.
\end{equation} 
Moreover, the  equality in the upper bound  is attained if and only if   $\Omega$ satisfies \eqref{optimal}, while the lower bound is asymptotically attained, as $L\to \infty$, by the sequence of slab-type domains $\Omega_L $ given by \eqref{slab}.
\end{thm}

In order to prove Theorem \ref{reverseholder}, we start with a preliminary estimate which, in particular, implies that, for every bounded  convex open set $\Omega\subseteq \mathbb R^N$ there exists a   polynomial $p_{\Omega}$ of degree  $N-1$ (whose coefficients depend on $\Omega$) such that 
 $$|\Omega|=P(\Omega)   r_{\Omega}  p_{\Omega}(r_{\Omega}).$$
 
\begin{lemma}\label{lem:alpha-upper}
Let $\Omega \subset \mathbb{R}^N$ be a bounded convex open set. Then there exists $\alpha=\alpha(\Omega)\in \big(0,r_{\Omega}^{-1}\big]$ satisfying
\begin{equation}\label{integalpha}
\frac{|\Omega|}{P(\Omega)} = \frac{1-(1-\alpha r_{\Omega})^N}{N\alpha},
\end{equation}
such that, for every $\beta\ge 0$, it holds
\begin{equation}\label{upperdistanzabeta}
\int_{\Omega} d_{\Omega}^{\beta}\,dx \le P(\Omega)\int_0^{r_{\Omega}} t^\beta(1-\alpha t)^{N-1}\,dt.
\end{equation}
\end{lemma}

\begin{proof}
We first assume $P(\Omega)=1$ and show that there exists $\alpha\in(0,r_\Omega^{-1}]$ such that
\begin{equation}\label{integalpha1}
|\Omega|=\frac{1-(1-\alpha r_{\Omega})^N}{N\alpha}.
\end{equation}
If $N=2$, it is sufficient to take
\[
\alpha=2\left(\frac1{r_{\Omega}}-\frac{|\Omega|}{r_{\Omega}^2}\right),
\]
which belongs to $\big(0,r_\Omega^{-1}\big]$ by \eqref{pervol}.
If $N>2$, consider
\[
f(t)=\frac{1-(1-tr_{\Omega})^N}{Nt}=\int_0^{r_{\Omega}}(1-st)^{N-1}\,ds,\qquad t\in\big(0,r_\Omega^{-1}\big].
\]
Then $f$ is continuous on $\big(0,r_\Omega^{-1}\big]$, differentiable on $\big(0,r_\Omega^{-1}\big)$, and
\[
f'(t)=(N-1)\int_0^{r_{\Omega}}(-s)(1-st)^{N-2}\,ds<0\qquad \text{for }t\in\big(0,r_\Omega^{-1}\big),
\]
so $f$ is strictly decreasing. Moreover,
\[
f(r_\Omega^{-1})=\frac{r_\Omega}{N},\qquad \lim_{t\to0^+}f(t)=r_\Omega.
\]
By \eqref{pervol} we have $|\Omega|\in\bigl[\frac{r_\Omega}{N},\,r_\Omega\bigr)$, hence there exists a unique $\alpha\in\big(0,r_\Omega^{-1}\big]$ such that $f(\alpha)=|\Omega|$, i.e.\ \eqref{integalpha1} holds.

Define $\gamma(t)=1-\alpha t$, which is nonnegative on $[0,r_\Omega]$. By the coarea formula,
\begin{equation}\label{integ}
\int_0^{r_\Omega}(\gamma(t))^{N-1}\,dt=f(\alpha)=|\Omega|=\int_0^{r_\Omega}L(t)\,dt,
\end{equation}
where $L(t)=P(\Omega(t))$ with $\Omega(t)$  defined by \eqref{parallelset}.
Since $t\mapsto L(t)^{1/(N-1)}$ is concave on $[0,r_\Omega]$, the function
\[
F(t)=L(t)^{1/(N-1)}-\gamma(t)
\]
is concave, with $F(0)=0$ and
\[
F(r_\Omega)=L(r_\Omega)^{1/(N-1)}-\gamma(r_\Omega)=0-1+\alpha r_\Omega\le 0.
\]
Therefore there exists $t_0\in(0,r_\Omega]$ such that
\[
F(t)\ge 0 \ \text{for }t\in(0,t_0],\qquad F(t)\le 0 \ \text{for }t\in[t_0,r_\Omega].
\]
Equivalently, setting $\Phi(t):=L(t)-(\gamma(t))^{N-1}$, we have $\Phi(t)\ge 0$ on $(0,t_0]$ and $\Phi(t)\le 0$ on $[t_0,r_\Omega]$.
Using again the coarea formula, for every $\beta\ge 0$ we can write
\begin{equation}\label{stimaintdis}
\int_\Omega d_\Omega^\beta\,dx
=\int_0^{r_\Omega} t^\beta L(t)\,dt
=\int_0^{r_\Omega} t^\beta \Phi(t)\,dt+\int_0^{r_\Omega} t^\beta (\gamma(t))^{N-1}\,dt.
\end{equation}
Since $t^\beta\le t_0^\beta$ on $[0,t_0]$ and $t^\beta\ge t_0^\beta$ on $[t_0,r_\Omega]$, and $\Phi$ changes sign accordingly, we obtain
\[
\int_0^{r_\Omega} t^\beta \Phi(t)\,dt
\le t_0^\beta\int_0^{r_\Omega}\Phi(t)\,dt.
\]
By \eqref{integ} we have $\int_0^{r_\Omega}\Phi(t)\,dt=0$, hence  \eqref{stimaintdis} implies 
\[
\int_\Omega d_\Omega^\beta\,dx \le \int_0^{r_\Omega} t^\beta(1-\alpha t)^{N-1}\,dt,
\]
which is \eqref{upperdistanzabeta} when $P(\Omega)=1$.

The general case follows by a scaling  argument. For $t=(P(\Omega))^{-1/(N-1)}$ set $\Omega'=t\Omega$, so that $P(\Omega')=1$ and $r_{\Omega'}=t r_\Omega$. Moreover,
\[
|\Omega'|=t^N|\Omega|,\qquad d_{\Omega'}(x)=t\,d_\Omega(x/t).
\]
Applying the previous step to $\Omega'$ and then  scaling back,  we get  \eqref{upperdistanzabeta}, with \[\alpha(\Omega)=t\,\alpha(\Omega')< \frac t {r_{\Omega'}}= r^{-1}_{\Omega},\] and \eqref{integalpha} follows from \eqref{integalpha1} by the same scaling.
\end{proof}

\begin{lemma}\label{lem:step1}
Let $N\ge 2$, $r>0$ and $0<\alpha\le r^{-1}$. Then
\begin{equation}\label{step1}
\int_0^{r} t (1-\alpha t)^{N-1}\,dt\leq \frac{N}{N+1} \left(\int_0^{r}  (1-\alpha t)^{N-1}\,dt\right)^2.
\end{equation}
Moreover, equality  in \eqref{step1} holds if and only if $\alpha=r^{-1}$.
\end{lemma}

\begin{proof}
With the change of variables $s=1-\alpha t$, we have
\[
\int_0^{r} t (1-\alpha t)^{N-1}\,dt
= \frac{1}{\alpha^2}\int_{1-\alpha r}^{1}\bigl(s^{N-1}-s^N\bigr)\,ds
=\frac{1}{\alpha^2}\left(\frac{1-(1-\alpha r)^N}{N}-\frac{1-(1-\alpha r)^{N+1}}{N+1}\right),
\]
while
\[
\left(\int_0^{r} (1-\alpha t)^{N-1}\,dt\right)^2
=\left(\frac{1-(1-\alpha r)^N}{N\alpha}\right)^2.
\]
Setting $x=1-\alpha r\in[0,1)$, inequality \eqref{step1} is equivalent to
\[
\frac{1-x^N}{N}-\frac{1-x^{N+1}}{N+1}\le \frac{1}{N(N+1)}(1-x^N)^2.
\]
Define
\[
F_N(x):=(1-x^N)^2-(N+1)(1-x^N)+N(1-x^{N+1}).
\]
Then, inequality \eqref{step1} holds if and only if $F_N(x)\ge 0$. A direct computation gives
\[
F_N(x)=x^{2N}+(N-1)x^N-Nx^{N+1}=x^N\,g_N(x),
\qquad \hbox{with } \ g_N(x)=N-1+x^N-Nx.
\]
For $x\in(0,1)$, we have that 
\[
g_N'(x)=N(x^{N-1}-1)<0,
\]
so $g_N$ is strictly decreasing on $[0,1)$ and, since $g_N(1)=0$, we have $g_N(x)>0$ for $x\in[0,1)$. Therefore $$F_N(x)>F_{N}(0)\quad   \hbox{ on } (0,1)\,, \qquad \hbox{ and } \qquad F_N(x)=0 \Longleftrightarrow x\in \{0,1\}.$$ 
Hence \eqref{step1} holds  with strict inequality for $0<\alpha<r^{-1}$ and is attained if and only if $\alpha= r^{-1}$.
\end{proof}

We are now in a position to prove Theorem \ref{reverseholder}.

\begin{proof}[Proof of Theorem \ref{reverseholder}] 
Let $\Omega\subset\mathbb{R}^N$ be a convex bounded open set. 
Since $d_\Omega$ is concave on $\Omega$, applying Theorem \ref{theo.borell} to $f=d_\Omega$ with $q=1$ and $s=\beta\ge 1$, we obtain
\begin{equation}\label{beta-norma}
\left(\fint_\Omega d_\Omega^\beta\,dx\right)^{1/\beta}\le C_{\beta,1}\,\fint_\Omega d_\Omega\,dx,\qquad \forall \beta\ge 1,
\end{equation}
where
\[
C_{\beta,1}=\binom{N+1}{N}\binom{N+\beta}{N}^{-1/\beta}=(N+1)\binom{N+\beta}{N}^{-1/\beta}.
\]
Using \eqref{upperdistanzabeta} with $\beta=1$ and Lemma \ref{lem:step1} (with $r=r_\Omega$), we infer
\begin{equation}\begin{split}\label{mercol0}
\fint_\Omega d_\Omega\,dx
&=\frac1{|\Omega|}\int_\Omega d_\Omega\,dx
\le \frac{P(\Omega)}{|\Omega|}\int_0^{r_\Omega} t(1-\alpha t)^{N-1}\,dt\\
&\le \frac{N}{N+1}\frac{P(\Omega)}{|\Omega|}\left(\int_0^{r_\Omega}(1-\alpha t)^{N-1}\,dt\right)^2,
\end{split}\end{equation}
where $\alpha=\alpha(\Omega)$ is given by Lemma \ref{lem:alpha-upper}. By \eqref{integalpha} we have
\begin{equation}\label{mercol}
\int_0^{r_\Omega}(1-\alpha t)^{N-1}\,dt=\frac{|\Omega|}{P(\Omega)},
\end{equation}
hence
\[
\fint_\Omega d_\Omega\,dx \le \frac{N}{N+1}\frac{|\Omega|}{P(\Omega)}.
\]
Plugging this into \eqref{beta-norma} yields
\[
\left(\fint_\Omega d_\Omega^\beta\,dx\right)^{1/\beta}
\le \binom{N+\beta}{N}^{-1/\beta}\,\frac{N|\Omega|}{P(\Omega)},
 \]
which concludes the proof of the upper bound in \eqref{lowerdistanzabeta1}.

Moreover, if a a convex bounded open set $\Omega\subseteq \mathbb R^N$ satisfies \eqref{optimal}, then Proposition \ref{lower-estim} implies
\begin{equation}\label{upperbound}
\fint_\Omega d_\Omega^\beta\,dx
=\binom{N+\beta}{N}^{-1}\left(\frac{N|\Omega|}{P(\Omega)}\right)^\beta,
\end{equation}
so  the upper bound in  \eqref{lowerdistanzabeta1} is attained. Conversely, if  \eqref{upperbound}  holds, then \eqref{beta-norma}, \eqref{mercol0} and  \eqref{mercol} yield  
\[\int_0^{r_\Omega} t(1-\alpha t)^{N-1}\,dt
= \frac{N}{N+1}\left(\int_0^{r_\Omega}(1-\alpha t)^{N-1}\,dt\right)^2.\]
Thanks to Lemma \ref{lem:step1}, this implies that $\alpha=r_{\Omega}^{-1}$ and, by using \eqref{integalpha}, we obtain  that $\Omega$ satisfies \eqref{optimal}.

Now we prove the lower bound in \eqref{lowerdistanzabeta1} when $\beta\geq 1$. 
We recall that, thanks to \cite[Corollary 6.1]{BPZ2}), it holds  \[
\lim_{p\to \infty} \big(\lambda_{p,\beta}(\Omega)\big)^{\frac{1}{p}}=\frac{1}{\| d_{\Omega} \|_{L^\beta(\Omega)}}.
\]
In particular,  \begin{equation*}
\label{eq:pitoinfty}
\lim_{p\to \infty}\pi_{p,\beta} =\lim_{p\to \infty}(\lambda_{p,\beta}((0,1)))^{1/\beta}= \left(\int_0^{\frac 1 2} x^\beta dx+ \int_{\frac 1 2}^1 (1-x)^\beta dx \right)^{-\frac 1 \beta}=2(\beta+1)^{1/\beta}.
\end{equation*}
Hence,   sending  $p\to \infty$ in   the inequality 
\[\left(\lambda_{p,\beta}(\Omega)\right)^{1/p}\frac{|\Omega|^{1-\frac{1}{p}+\frac{1}{\beta}}}{P(\Omega)}\leq\frac{\pi_{p,\beta}}{2} ,\]
proved in  \cite[Main Theorem]{Bra1}, 
we easily  get that

\[ \|d_{\Omega} \|_{L^\beta(\Omega)} \geq \frac 1 {\left(\beta+1\right)^{1/\beta}}\frac {|\Omega|^{1+\frac{1}{\beta}}}{P(\Omega)} ,\]
that is the lower bound in \eqref{lowerdistanzabeta1}. 
Finally, thanks to   the  asymptotic behaviour
		\[  \int_{\Omega_L} d_{\Omega_{L}}^{\beta } \, dx \sim L^{N-1} \, \left(\frac{1}{2}\right)^{\beta} \,\frac{1}{\beta+1}, \qquad \mbox{ as } L \to \infty\]
		 and 
		\[ P(\Omega_L) \sim 2 \, L^{N-1} \qquad \mbox{ and } \qquad |\Omega_L| \sim  L^{N-1}, \qquad \mbox{ as } L \to \infty,\] 
		(see \cite[Proof of Theorem 1.1]{PZ}), we easily get that 
	 \[\lim_{L\to \infty}\left(\frac{P(\Omega_L)}{|\Omega_L|}\right)^\beta \fint_{\Omega} d_{\Omega}^{\beta}\,dx  = \frac{1}{\beta+1},  \]
i.e. the estimate from below in \eqref{lowerdistanzabeta1} is sharp.
\end{proof}

\section{Proof of the $N$-dimensional Makai inequality}\label{proof}

We start with the following preliminary result that will be useful to prove the sharpness of the inequality \eqref{eq:lower bound per}. The argument follows the lines of \cite[Proposition 4.1]{BBPtorsion}, proved in the cases $1=q<p<\infty$ and $1<p=q<\infty$, and we include it for the reader's convenience.

\begin{prop}\label{stimethin}
Let $A\subset\mathbb{R}^{N-1}$ be an open set with finite $\mathcal{H}^{N-1}$-measure and $h\in Lip(A)\cap L^{\infty}(A)$ with $h>0$ on $A$. For every $\varepsilon>0$, let $\Omega_\varepsilon$ be defined by 
\begin{equation}\label{thin}
\Omega_\varepsilon:=\big\{(x,y)\in A\times \mathbb{R}\ :\ 0<y<\varepsilon h(x)\big\}.\end{equation}
Then, for every  $1 \le q < p<\infty$, we have
\begin{equation}\label{asythin}
\limsup_{\varepsilon\to0} \varepsilon^{\frac{pq+p-q}{q}} \lambda_{p,q}(\Omega_\varepsilon) \le \frac{
\pi_{p,q}^{\,p}}{
\left(\displaystyle\int_{A}h^{\frac{pq}{p-q}+1} \,dx \right)^{\frac{p}{q}-1}}.
\end{equation}
\end{prop}

\begin{proof}  
Let $u\in W_0^{1,p}(0,1)$ be the positive minimizer in the definition of $\pi_{p,q}$. For brevity we set $\beta:=\frac{pq}{p-q}$ and observe that the positive function $w\in W_0^{1,p}(0,1)$ given by $w:=(\pi_{p,q})^{-\frac{\beta}{q}}u$ satisfies 
\begin{equation*}\label{normw}
\int_0^1 w^q\,dt
\;=\;\int_0^1 |w'|^p\,dt
\;=\;\left(\pi_{p,q}\right)^{-\beta},
\end{equation*}
moreover,  $w$  belongs to  $C^1([0,1])$ (see \cite[Theorem 3.1]{DM}). For every $\varepsilon>0$, we define
\begin{equation*}\label{defweps}
w_\varepsilon(x,y)
:= 
\bigl(\varepsilon h(x)\bigr)^{\frac{\beta}{q}}\,
w\!\left(\frac{y}{\varepsilon\, h(x)}\right),
\qquad (x,y)\in \Omega_\varepsilon.
\end{equation*}
For every $x\in A$, the change of variables $t=\frac{y}{\varepsilon\,h(x)}$
yields
\begin{equation}\label{scaling}
\int_{0}^{\varepsilon h(x)} w_\varepsilon^q(x,y)\,dy
\;=\,(\varepsilon h(x))^{\beta+1}
\; \left(\pi_{p,q}\right)^{-\beta}, \qquad \hbox{for every } x\in A,
\end{equation}
and, since 
\[
\partial_y w_\varepsilon(x,y)
= (\varepsilon h(x))^{\frac{\beta}{p}}\,
w'\!\bigl(\tfrac{y}{\varepsilon h(x)}\bigr),
\]
we have
\begin{equation}\label{scaling-grady}
\int_{0}^{\varepsilon h(x)} |\partial_y w_\varepsilon(x,y)|^p\,dy
\;=\; (\varepsilon h(x))^{\beta+1}
\: \left(\pi_{p,q}\right)^{-\beta}.
\end{equation}
We now estimate the contribution of $\nabla_x w_\varepsilon$ in order to show that
 \begin{equation}\label{gradx}\int_0^{\varepsilon h(x)} |\nabla_x w_\varepsilon|^p\,dy=o(\varepsilon^{\beta+1}), \qquad \hbox{ as } \varepsilon\to 0,\end{equation}
 uniformly with respect to $x\in A$.
To this end, we note that, since $h$ is Lipschitz continuous, there exists a  $\mathcal H^{N-1}$-negligible subset $N\subseteq A$,  depending  only on $h$,  such that, for every $x\in A\setminus N$ and $0< y <\varepsilon h(x)$, using the notation $t=\frac{y}{\varepsilon\,h(x)}$ and $M := \max \big\{\|\nabla h\|_{L^\infty(A)}, \| h\|_{L^\infty(A)}\big\}$, it holds 
\[
 \begin{split}
 |\nabla_x w_\varepsilon(x,y)| = 
 	& \left|  \varepsilon^{\frac{\beta}{q}} h(x)^{\frac{\beta}{q}-1} \nabla h(x) \left( \frac{\beta}{q}w\!\left(t\right) - t \,
w'\!\left(t\right) \right) \right| \\
\leq &M^{\frac{\beta}{q}} \, \varepsilon^{\frac{\beta}{q}} \left|\left( \frac{\beta}{q}w\!\left(t\right) - t \,
w'\!\left(t\right) \right) \right| \\
\leq & \frac{\beta}{q} M^{\frac{\beta}{q}} \, \varepsilon^{\frac{\beta}{q}} \bigl(|w(t)| + |w'(t)|\bigr)
 \end{split}
\]
where  the last estimate follows since $t < 1 \leq \frac{\beta}{q}$.
Raising the previous estimate to the power $p$ and integrating in $y$, using the same change of variable as before, we get
\[
\begin{split}
 \int_0^{\varepsilon h(x)} |\nabla_x w_\varepsilon|^p\,dy 
\le &  \left(\frac{\beta}{q}\right)^p M^{p+\beta+1}\,\varepsilon^p\;
\varepsilon^{\beta+1} \int_0^1 \bigl(|w| + |w'|\bigr)^p\,dt.
\end{split}
\]
Therefore, observing that 
\begin{equation*}\begin{split} \int_0^1 \bigl(|w| + |w'|\bigr)^p\,dt&\leq 2^{p-1}\left( \|w\|^{p-q}_{L^\infty(0,1)}\int_0^1 |w|^q\ dt\,+ \int_0^1 |w'|^p\,dt\right)\\
&= 2^{p-1} \pi_{p,q}^{-\beta} \left( \|w\|^{p-q}_{L^\infty(0,1)}+ 1\right), 
\end{split}
\end{equation*}
we get 
\[
 \int_0^{\varepsilon h(x)} |\nabla_x w_\varepsilon|^p\,dy \le C\,\varepsilon^p\;
\varepsilon^{\beta+1}
\]
where $C$ is a positive constant depending on $p,q,w,M$, that is, \eqref{gradx} holds  uniformly with respect to $x\in A\setminus N$.

Let now $\phi\in C_0^\infty(A)$ and set
$v_{\varepsilon}(x,y):=\phi(x)\,w_\varepsilon(x,y)$. Then $v_{\varepsilon}\in W_0^{1,p}(\Omega_\varepsilon)$ and
\begin{equation}\label{rayleigh}
\lambda_{p,q}(\Omega_\varepsilon)\le
\frac{\displaystyle\int_{\Omega_\varepsilon} |\nabla v_{\varepsilon}|^p\,dx\,dy}
{\left(\displaystyle\int_{\Omega_\varepsilon}|v_{\varepsilon}|^q\,dx\,dy\right)^{p/q}}.
\end{equation}
By Fubini's Theorem and \eqref{scaling}, the denominator equals
\begin{equation}\label{denom}
\int_{\Omega_\varepsilon}|v_{\varepsilon}|^q\,dx\,dy
= \left(\pi_{p,q}\right)^{-\beta}\,\varepsilon^{\beta+1}
\int_A |\phi(x)|^q\,h(x)^{\beta+1}\,dx.
\end{equation}
For the numerator, we write
$\nabla v_{\varepsilon} = (\phi\,\nabla_x w_\varepsilon + w_\varepsilon\,\nabla\phi,\;\phi\,\partial_y w_\varepsilon).$
We observe that
\[|w_\varepsilon(x,y)|\leq  
\bigl(\varepsilon \|h\|_{L^\infty(A)} \bigr)^{\frac{\beta}{q}}\,\| w\|_{L^\infty(0,1)}
\qquad (x,y)\in \Omega_\varepsilon.
\]
Since $p>q$, applying  the previous inequality and \eqref{scaling}, we obtain that 
\begin{equation*}\begin{split}
\int_0^{\varepsilon h(x)}  w_\varepsilon^p\,dy=\int_0^{\varepsilon h(x)}  w_\varepsilon^{p-q}w_\varepsilon^q \,dy&\leq \left((\varepsilon \|h\|_{L^\infty(A)} \bigr)^{\frac{\beta}{q}}\,\| w\|_{L^\infty(0,1)} \right)^{p-q}  \int_0^{\varepsilon h(x)}  w_\varepsilon^{q}\,dy\\
&\leq \left((\varepsilon \|h\|_{L^\infty(A)} \bigr)^{\frac{\beta}{q}}\,\| w\|_{L^\infty(0,1)} \right)^{p-q} (\varepsilon \|h\|_{L^\infty(A)})^{\beta+1}
\; \left(\pi_{p,q}\right)^{-\beta},\\
&=C \varepsilon^{p+\beta+1},
 \end{split}
\end{equation*}
with a constant which does not depend on $x$.
This yields  \begin{equation}
\label{normap}\int_0^{\varepsilon h(x)}  w_\varepsilon^p\,dy =o(\varepsilon^{\beta+1}) \qquad \hbox{ as } \varepsilon\to 0,
\end{equation}
 uniformly with respect to $x\in A \setminus N$.
We claim that 
\begin{equation}
\label{CLAIM}\int_0^{\varepsilon h(x)} |\nabla v_\varepsilon|^p\,dy \sim |\phi(x)|^p \; \varepsilon^{\beta+1}\,h(x)^{\beta+1} 
\: \left(\pi_{p,q}\right)^{-\beta} \qquad \hbox{ as } \varepsilon\to 0,
\end{equation}
 uniformly with respect to $x\in A\setminus N$.
Indeed, for $1<p\leq 2$, thanks to  the subadditivity of the function $[0,+\infty)\ni t\to t^{\frac p 2 }$,  
 we obtain  that 
\begin{equation*}\begin{split} 
\int_0^{\varepsilon h(x)} |\nabla v_\varepsilon|^p\,dy 
& =\int_0^{\varepsilon h(x)} \left(|\phi\,\nabla_x w_\varepsilon + w_\varepsilon\,\nabla\phi|^2+|\phi\partial_y w_\varepsilon|^2\right)^{\frac{2}{p}}\,dy\\
&\leq \int_0^{\varepsilon h(x)} |\phi\,\nabla_x w_\varepsilon + w_\varepsilon\,\nabla\phi|^p+ |\phi(x)|^p  |\partial_y w_\varepsilon|^p\,dy\\
&\leq 2^{p-1} |\phi(x)|^p  \int_0^{\varepsilon h(x)}    |\nabla_x w_\varepsilon|^p  \,dy  + 2^{p-1} |\nabla\phi(x)|^p  \int_0^{\varepsilon h(x)}  w_\varepsilon^p\,dy\  \\
&+ |\phi(x)|^p \int_0^{\varepsilon h(x)}  |\partial_y w_\varepsilon|^p\,dy.
\end{split}
\end{equation*}
Hence,  using \eqref{gradx},  \eqref{normap} and  \eqref{scaling-grady}, the above chain of inequalities implies \eqref{CLAIM}.
If instead $p\geq 2$, by  applying the   Minkowski inequality,     we have   that 
\begin{equation*}\begin{split}
\left(  \int_0^{\varepsilon h(x)} |\nabla v_\varepsilon|^p\,dy\right)^{\frac{2}{p}}
= & \left(  \int_0^{\varepsilon h(x)} \left(|\phi\,\nabla_x w_\varepsilon + w_\varepsilon\,\nabla\phi|^2+|\phi\partial_y w_\varepsilon|^2\right)^{\frac{p}{2}}\,dy\right)^{\frac{2}{p}}\\
\leq & \left(\int_0^{\varepsilon h(x)} |\phi\,\nabla_x w_\varepsilon + w_\varepsilon\,\nabla\phi|^p dy\right)^{\frac{2}{p}}+ |\phi(x)|^2\left(\ \int_0^{\varepsilon h(x)}\  |\partial_y w_\varepsilon|^p\,dy \right)^{\frac{2}{p}}\\
\leq &\left( |\phi(x)| \left(\int_0^{\varepsilon h(x)} |\nabla_x w_\varepsilon|^p dy\right)^{\frac{1}{p}}+ |\nabla\phi(x)|\left(\int_0^{\varepsilon h(x)}  |w_\varepsilon |^p dy \right)^{\frac{1}{p}}  \right)^{2} \\
& + |\phi(x)|^2\left(\ \int_0^{\varepsilon h(x)}\  |\partial_y w_\varepsilon|^p\,dy \right)^{\frac{2}{p}}.
\end{split}
\end{equation*}
By using again  \eqref{gradx},  \eqref{normap} and  \eqref{scaling-grady},  we obtain that   \eqref{CLAIM} holds also in this case.
Since
\begin{equation*}\label{numer}
\int_{\Omega_\varepsilon} |\nabla v_{\varepsilon}|^p\,dx\,dy=\int_A dx  \int_0^{\varepsilon h(x)} |\nabla v_\varepsilon|^p\,dy, 
\end{equation*}
\eqref{CLAIM} and   the Lebesgue theorem imply that
\begin{equation}\label{numer}
\int_{\Omega_\varepsilon} |\nabla v_{\varepsilon}|^p\,dx\,dy
= \left(\pi_{p,q}\right)^{-\beta}\,\varepsilon^{\beta+1}
\int_A |\phi(x)|^p\,h(x)^{\beta+1}\,dx
\;\bigl(1 + o(1)\bigr) \qquad \hbox{ as } \varepsilon\to 0.
\end{equation}
Substituting \eqref{denom} and \eqref{numer} into \eqref{rayleigh}, we obtain
\[
\varepsilon^{p+\frac{p}{\beta}}\,\lambda_{p,q}(\Omega_\varepsilon)
\le
\pi_{p,q}^{\,p}\,
\frac{\displaystyle\int_A |\phi|^p\,h^{\beta+1}\,dx}
{\left(\displaystyle\int_A |\phi|^q\,h^{\beta+1}\,dx\right)^{p/q}}
\;\bigl(1 + o(1)\bigr),
\]
Passing to the $\limsup$ as $\varepsilon\to 0$ yields
\begin{equation}\label{limfi}
\limsup_{\varepsilon\to 0}
\varepsilon^{p+\frac{p}{\beta}}\,\lambda_{p,q}(\Omega_\varepsilon)
\le \pi_{p,q}^{\,p}\,
\frac{\displaystyle\int_A |\phi|^p\,h^{\beta+1}\,dx}
{\left(\displaystyle\int_A |\phi|^q\,h^{\beta+1}\,dx\right)^{p/q}},
\qquad \forall\,\phi\in C_0^\infty(A).
\end{equation}
Finally, let $\{\phi_k\}\subset C_0^\infty(A)$ satisfy
$0\le\phi_k\le 1$ and $\phi_k\to 1_A$ pointwise.
Since $h^{\beta+1}\in L^1(A)$,   we get, for $r\in\{p,q\}$,  as $k\to \infty$, 
\[\int_A|\phi_k|^r\,h^{\beta+1}\,dx\to\int_A h^{\beta+1}\,dx, \]
 so that the quotient in the right-hand side of \eqref{limfi} converges to
$\bigl(\int_A h^{\beta+1}\,dx\bigr)^{1-p/q}$,
which implies \eqref{asythin}. \end{proof}

We are now in a position to prove the main result of the paper. 

\begin{proof}[Proof of Theorem \ref{cor:upperb}]  By combining \eqref{eq:prizag} with the inequality \eqref{lowerdistanzabeta1},
	applied with  $\beta=\frac{pq}{p-q}\geq 1$, we get
	\begin{equation*}\begin{split}
\lambda_{p,q}(\Omega) &\ge \frac{\left(\displaystyle \frac{\pi_{p,q}}{2} \right)^p \left( \displaystyle \frac 1 {\beta+1}  \binom{N+\beta}{N}\right)^{\frac p \beta} }{\left(\displaystyle \frac{N|\Omega|}{P(\Omega)}\right)^p |\Omega|^{\frac p \beta} }\\
&= \left(\displaystyle \frac{\pi_{p,q}}{2N} \right)^p \left( \frac{(\beta+2)(\beta+3)\cdots (\beta+N)}{N!}\right)^{\frac p \beta} \left( \frac{P(\Omega)}{|\Omega|^{(\beta+1)/\beta}}\right)^p\\
&= \left(\displaystyle \frac{\pi_{p,q}}{2N} \right)^p \prod_{j=2}^N\left(1+\frac{\beta}{j}\right)^{\frac p \beta} \left( \frac{P(\Omega)}{|\Omega|^{(\beta+1)/\beta}}\right)^p\\
&=K_{p,q,N} \left(\frac{P(\Omega)}{\,|\Omega|^{1-\frac 1 p +\frac 1 q}}\right)^p
\end{split}
\end{equation*}
which proves  the desired lower bound  \eqref{eq:lower bound per}. 

In order to show that \eqref{eq:lower bound per} is sharp, it is sufficient to construct a family   $\{\Omega_{\varepsilon}\}_{\varepsilon}$ of bounded convex open sets such that
\[\limsup_{\varepsilon\to 0} \lambda_{p,q}(\Omega_\varepsilon) 
\left(\frac{|\Omega_{\varepsilon}|^{1-\frac 1 p +\frac 1 q}}  {P(\Omega_{\varepsilon})}\right)^p \leq K_{p,q,N}.\]
With the notation of Proposition \ref{stimethin}, we choose $A=B^{N-1}_1$, where $B^{N-1}_1$ is the unit open ball of $\mathbb R^{N-1}$,   $h(x)=1-|x|$ and   we consider the family $\{\Omega_{\varepsilon}\}_{\varepsilon}$ defined by \eqref{thin}. Then, as $\varepsilon\to 0$, we have that 
\[\frac{P(\Omega_{\varepsilon})}{\,|\Omega_{\varepsilon}|^{1-\frac 1 p +\frac 1 q}}\sim \frac{2 \mathcal{H}^{N-1}(A)}{ \left( \varepsilon \displaystyle\int_A h \,dx \right)^{1-\frac{1}{p} +\frac{1}{q}}}.
\]
Combining this fact with \eqref{asythin} and taking into account that  $\mathcal H^{N-1}(A)=\omega_{N-1}$,  we have 
\begin{equation}\label{gio1}\begin{split}
\limsup_{\varepsilon\to 0} \lambda_{p,q}(\Omega_\varepsilon)
\left(\frac{|\Omega_{\varepsilon}|^{1-\frac 1 p +\frac 1 q}}{P(\Omega_{\varepsilon})}\right)^p
&\le \left(\frac{\pi_{p,q}}{2\omega_{N-1} }\right)^p
\frac{\left(\displaystyle\int_A h\,dx\right)^{p-1+\frac{p}{q}}}
{\left(\displaystyle\int_A h^{\frac{pq}{p-q}+1}\,dx\right)^{\frac{p}{q}-1}}.
\end{split}
\end{equation}
 For $A=B_1^{N-1}$,    $\beta=\frac{pq}{p-q}$, using polar coordinates on the function $h(x)=1-|x|$, we compute  \begin{equation}\label{gio2}
\int_A h(x)\,dx=(N-1)\omega_{N-1}\int_0^1 r^{N-2}(1-r)\,dr=(N-1)\omega_{N-1}\,B(N-1,2),
\end{equation}
and 
\begin{equation}\label{gio3}
\int_A h(x)^{\beta+1}\,dx=(N-1)\omega_{N-1}\int_0^1 r^{N-2}(1-r)^{\beta+1}\,dr=(N-1)\omega_{N-1}\,B(N-1,\beta+2).
\end{equation}
Joining \eqref{gio1}-\eqref{gio3}, we conclude
\begin{equation*}\begin{split}
\limsup_{\varepsilon\to 0} \lambda_{p,q}(\Omega_\varepsilon)
\left(\frac{|\Omega_{\varepsilon}|^{1-\frac 1 p +\frac 1 q}}{P(\Omega_{\varepsilon})}\right)^p
&\le \left(\frac{\pi_{p,q}}{2}\right)^p (N-1)^p
\frac{ \bigl(B(N-1,2)\bigr)^{p-1+\frac{p}{q}}}{\bigl(B(N-1,\beta+2)\bigr)^{\frac{p}{q}-1}}\\
&=\left(\frac{\pi_{p,q}}{2}\right)^p (N-1)^p
\frac{\left(\frac{\Gamma(N-1)\Gamma(2)}{\Gamma(N+1)}\right)^{p-1+\frac{p}{q}}}
{\left(\frac{\Gamma(N-1)\Gamma(\beta+2)}{\Gamma(N+\beta+1)}\right)^{\frac{p}{q}-1}}\\
&=\left(\frac{\pi_{p,q}}{2N}\right)^p
\left(\frac{\Gamma(N+\beta+1)}{N!\,\Gamma(\beta+2)}\right)^{\frac{p}{q}-1}\\
&=\left(\frac{\pi_{p,q}}{2N}\right)^p
\prod_{j=2}^N\left(1+\frac{\beta}{j}\right)^{\frac{p}{q}-1}
=K_{p,q,N}.
\end{split}
\end{equation*}
This completes the proof of the theorem.
\end{proof}

We conclude the section with the proof of Corollary \ref{maincoro}.
\begin{proof}[Proof of Corollary \ref{maincoro}] By applying 
\eqref{eq:lower bound per}  in the case $q=1$ and taking into account that $$\pi_{p,1}^p=\frac{1}{T_p(0,1)}=2^p(p'+1)^{p-1},$$ (see, e.g., \cite[Remark 2.4]{Bra1}), we obtain  that $$T_{p}(\Omega) \le N^p \left( \prod_{j=1}^N \left(\displaystyle  \frac{p'}{j}+1\right) \right)^{1-p}  \frac{\,|\Omega|^{2p-1}}{(P(\Omega))^p}=N^p \binom{N+p'}{N}^{1-p} \frac{\,|\Omega|^{2p-1}}{(P(\Omega))^p}.$$
The proof is over. \end{proof}

\section{Other geometric sharp bounds for the Poincar\'e--Sobolev constants}\label{AppB}
First of all,  as a byproduct of Proposition \ref{upper-estim} and the Makai-type inequality \eqref{eq:prizag}, in the next corollary we  give  a sharp lower bound for $\lambda_{p,q}(\Omega)$ in terms of $r_\Omega$ and $P(\Omega)$. We note that  this result generalizes \cite[Corollary 6.5]{bradua}, proved for $p=2$ by using a different argument.

\begin{cor}\label{cor:inradius-perimeter}
Let $N\ge 2$. Let $1\le q<p<\infty$ and let $\Omega\subset\mathbb{R}^N$ be a convex bounded open set. Then the following sharp lower bound holds
\begin{equation}\label{eq:lower boundinra}
\lambda_{p,q}(\Omega)\ge \left(\frac{\pi_{p,q}}{2}\right)^p
\left(\frac{1}{pq+p-q}\right)^{\frac{p-q}{q}}
\frac{1}{r_{\Omega}^{\frac{pq+p-q}{q}}\, P(\Omega)^{\frac{p-q}{q}}}.
\end{equation}
Moreover,  the equality in  \eqref{eq:lower boundinra} is asymptotically attained, as $L\to\infty$, by the sequence
$\Omega_L$  given by \eqref{slab}.
\end{cor}
We explicitly note that, as $q\nearrow p$, the lower bound \eqref{eq:lower boundinra}  implies the Hersch--Protter inequality
\[
\lambda_{p}(\Omega) \ge \left(\frac{\pi_{p}}{2}\right)^p \frac{1}{r_{\Omega}^p}.
\]

\begin{proof}
The inequality \eqref{eq:lower boundinra} follows by combining \eqref{eq:prizag} with \eqref{sup} applied with $\beta=\frac{pq}{p-q}$.
Moreover, the sharpness follows from the fact that both \eqref{eq:prizag} and \eqref{sup} are asymptotically attained, as $L\to\infty$, by the slab-type sequence $\Omega_L$ (see \cite[Theorem 1.1]{PZ}). \end{proof}

The remaining part of this section is devoted to establish two sharp upper bounds for $\lambda_{p,q}(\Omega)$.
 We recall that, in \cite[Theorem 5.7]{BPZ1} (see also  \cite[Theorem 1.1]{BM} in the case $p=2$), it is shown that 
\begin{equation*}
\label{HPBM}
\lambda_{p,q}(\Omega)\,|\Omega|^{\frac{p-q}{q}} \ge \left( \frac{\pi_{p,q}}{2} \right)^p\,\frac{1}{r_\Omega^p},
\end{equation*}
for every open bounded convex set $\Omega\subset\mathbb{R}^N$.

The following result represents the counterpart of the above  inequality. 
The proof follows the lines of that of \cite[Theorem 1.2]{BM}, proved when  $1\leq q<2=p$; we outline it here for the sake of completeness. 

\begin{thm}
\label{teo:upperbound}
Let $1\le q<p$. For every  bounded convex open set $\Omega\subset\mathbb{R}^N$, we have
\begin{equation}
\label{BM}
\lambda_{p,q}(\Omega)\,|\Omega|^\frac{p-q}{q}\le \frac{\omega_N^\frac{p-q}{q}\,\lambda_{p,q}(B_1)}{r_\Omega^p}.
\end{equation}
The equality is attained if and only if $\Omega$ is a ball.
\end{thm}

\begin{proof}
Up to translations, we may assume that $B_{r_\Omega}\subset \Omega$ is centered at the origin.
Let $u\in W^{1,p}_0(B_1)$ be a nonnegative minimizer for $\lambda_{p,q}(B_1)$. By symmetry of the ball, we can take $u$ radial, i.e.\ $u(x)=f(|x|)$ with $f\in C^1([0,1])$, $f$ decreasing and $f'(0)=0$.
In polar coordinates this yields
\begin{equation}\label{eq:ball-quotient}
\lambda_{p,q}(B_1)=\frac{N\omega_N\displaystyle\int_0^1 |f'(t)|^p\,t^{N-1}\,dt}
{\Bigl(N\omega_N\displaystyle\int_0^1 f(t)^q\,t^{N-1}\,dt\Bigr)^{\frac pq}}\,.
\end{equation}
Define the Minkowski functional (gauge) of $\Omega$ by
\[
j_\Omega(x):=\inf\{r>0:\ x\in r\Omega\}.
\]
Since $\Omega$ is convex and contains the origin, $j_\Omega$ is Lipschitz and $j_\Omega= 1$ on $\partial\Omega$; hence the composition
\[
\varphi(x):=f\bigl(j_\Omega(x)\bigr)
\]
belongs to $W^{1,p}_0(\Omega)$ and can be used as a competitor in the Rayleigh quotient for $\lambda_{p,q}(\Omega)$.
Thanks to  the identity of level sets \[
\{j_\Omega=t\}=t\,\partial\Omega \qquad \hbox{ for every }t>0, \]
and to  the $0$-homogeneity of $\nabla j_\Omega$, by  using  the coarea formula and  exploiting the change of variables $x=ty$, we get
\begin{align*}
\int_\Omega \varphi^q\,dx&=\int_0^1 f(t)^q\left(\int_{t\partial\Omega}\frac{1}{|\nabla j_\Omega|}\,d\mathcal H^{N-1}\right)\,dt\\
&=\left(\int_0^1 f(t)^q\,t^{N-1}\,dt\right)\left(\int_{\partial\Omega}\frac{1}{|\nabla j_\Omega|}\,d\mathcal H^{N-1}\right).
\end{align*}

Now, on $\partial\Omega$ one has the geometric identity
\begin{equation}
\label{scalar}
\frac{1}{|\nabla j_\Omega(x)|}=\langle x,\nu_\Omega(x)\rangle\qquad \text{for }\mathcal H^{N-1}\text{-a.e. }x\in\partial\Omega,
\end{equation}
(see, e.g., \cite[Lemma 2.3]{BM}), and 
so, applying the divergence theorem, it follows that
\begin{equation}\label{eq:Lq}
\int_\Omega \varphi^q\,dx
= N|\Omega|\int_0^1 f(t)^q\,t^{N-1}\,dt.
\end{equation}

Again by coarea and the same change of variables, we have  that 
\begin{equation}
\label{lunedi}
\int_\Omega |\nabla\varphi|^p\,dx
=\left(\int_0^1 |f'(t)|^p\,t^{N-1}\,dt\right)\left(\int_{\partial\Omega}|\nabla j_\Omega|^{p-1}\,d\mathcal H^{N-1}\right).
\end{equation}
Using \eqref{scalar}, we rewrite
\[
\int_{\partial\Omega}|\nabla j_\Omega|^{p-1}\,d\mathcal H^{N-1}
=\int_{\partial\Omega}\frac{1}{\langle x,\nu_\Omega\rangle^{p-1}}\,d\mathcal H^{N-1}.
\]
Since, by convexity and the definition of inradius, one has
\[
\langle x,\nu_\Omega(x)\rangle\ge r_\Omega\qquad\text{for }\mathcal H^{N-1}\text{-a.e. }x\in\partial\Omega,
\]
(see, e.g., \cite[Lemma 2.1]{BM}),
we get that 
\[
\frac{1}{\langle x,\nu_\Omega\rangle^{p-1}}
=\frac{\langle x,\nu_\Omega\rangle}{\langle x,\nu_\Omega\rangle^{p}}
\le \frac{\langle x,\nu_\Omega\rangle}{r_\Omega^{p}},
\]
which, used in \eqref{lunedi}, gives
\begin{equation}\label{eq:Dir}
\int_\Omega |\nabla\varphi|^p\,dx
\le \frac{N|\Omega|}{r_\Omega^p}\int_0^1 |f'(t)|^p\,t^{N-1}\,dt.
\end{equation}

Combining \eqref{eq:Lq}--\eqref{eq:Dir} and recalling \eqref{eq:ball-quotient} we  get
\[
\lambda_{p,q}(\Omega)\le
\frac{\displaystyle\int_\Omega |\nabla\varphi|^p\,dx}
{\Bigl(\displaystyle\int_\Omega \varphi^q\,dx\Bigr)^{\frac pq}}
\le
\frac{|\Omega|^{1-\frac pq}}{r_\Omega^p}\,
\frac{\displaystyle N|\Omega|\int_0^1 |f'|^p t^{N-1}dt}{\displaystyle \bigl(N|\Omega|\int_0^1 f^q t^{N-1}dt\bigr)^{\frac pq}}
=\frac{\omega_N^{\frac{p}{q}-1}\,|\Omega|^{1-\frac{p}{q}}}{r_\Omega^p}\,\lambda_{p,q}(B_1),
\]
which is equivalent to \eqref{BM}.

Now we  assume that equality holds in \eqref{BM}. Then equality must hold in the pointwise bound
$\langle x,\nu_\Omega(x)\rangle\ge r_\Omega$ for $\mathcal H^{N-1}$-a.e.\ $x\in\partial\Omega$, hence
\[
\langle x,\nu_\Omega(x)\rangle=r_\Omega \qquad\text{for }\mathcal H^{N-1}\text{-a.e.\ }x\in\partial\Omega.
\]
In particular, the competitor $\varphi=f\circ j_\Omega$ is optimal for $\lambda_{p,q}(\Omega)$ and thus it is a weak solution of
\[
-\Delta_p \varphi = \lambda_{p,q}(\Omega)\,\varphi^{\,q-1}\qquad \text{in }\Omega.
\]
By interior regularity, $\varphi\in C^{1,\alpha}_{\mathrm{loc}}(\Omega)$, hence $j_\Omega=f^{-1}\circ\varphi$ is $C^1$ in $\Omega\setminus\{0\}$, and the boundary identity above forces $\Omega$ to be a ball by the characterization of equality in the inradius inequality for convex sets (see, again, \cite[Lemma 2.1]{BM}). This concludes the proof. \end{proof}

\begin{cor}\label{cor:upper-inradius-perimeter}
Let $N\ge 2$. Let $1\le q<p<\infty$ and let $\Omega\subset\mathbb{R}^N$ be a convex bounded open set. Then the following sharp upper bound  holds:
\begin{equation}
\label{eq:upper boundinra}
\lambda_{p,q}(\Omega)\,\le \frac{(N\omega_N)^\frac{p-q}{q}\,\lambda_{p,q}(B_1)}{r_{\Omega}^{\frac{pq+p-q}{q}} P(\Omega)^{\frac{p-q}{q}}}.
\end{equation}
Moreover,  the equality in \eqref{eq:upper boundinra} is attained if and only if $\Omega$ is a ball. 

\end{cor}

\begin{proof}
\noindent The result is obtained  by combining Theorem \ref{teo:upperbound}  with the inequality
\[\frac{1}{|\Omega|}\leq \frac N {  r_{\Omega} P(\Omega) },\]
which derives  from   \eqref{pervol} and holds as equality on the balls. Indeed,   we have that 
\[\lambda_{p,q}(\Omega)\le \frac{\omega_N^\frac{p-q}{q}\,\lambda_{p,q}(B_1)}{ r_\Omega^p |\Omega|^\frac{p-q}{q}}\leq \frac{(N \omega_N)^\frac{p-q}{q}\,\lambda_{p,q}(B_1)}{ r_{\Omega}^{\frac{pq+p-q}{q}}\, P(\Omega)^{\frac{p-q}{q}}}.\]
 \end{proof}

\appendix
\section{Ill-posedness  in the class of simply connected open sets}\label{AppA}

In this section, we show that the minimization and maximization problems for the shape functional $\mathcal{F}_{p,q}$ defined in \eqref{prodotto} are, in general, both ill-posed, even when restricted to the class of simply connected open sets. 

More precisely, we first observe, in Proposition \ref{prop:infimum}, that, for $1\le q<p<\infty$ or $1<q=p<\infty$, the infimum is zero and then, in Proposition \ref{prop:supremum},  we show that, for $p=2$, the supremum of  $\mathcal{F}_{p,q}$ is infinity.  

\begin{prop}\label{prop:infimum}
Let $1 \leq q < p < \infty$  or $1<q=p<\infty$. Then
\begin{equation}\label{infsimply}
\inf \big\{\mathcal{F}_{p,q}(\Omega):   \Omega\subset \mathbb{R}^{N} \text{ simply connected open set }\big\}=0.
\end{equation}
\end{prop}

\begin{proof}
We exhibit a sequence of simply connected open sets along which the functional $\mathcal{F}_{p,q}$ tends to zero. To this end, let $\{\Omega_n\}_{n\in\mathbb{N}}$ be a sequence of smooth perturbations of the unit ball $B_1$ such that

\begin{equation}\label{perturbation}
B_{1/2}\subset\Omega_n\subset B_2\qquad\text{and}\qquad P(\Omega_n)\to+\infty \text{ as } n\to\infty.
\end{equation}

Such a sequence can be constructed, for instance, by considering increasingly oscillatory perturbations of the boundary while preserving simple connectivity (see, for example, \cite{BBP} for an explicit construction in dimension $N=2$).

By \eqref{perturbation}, we have that $|\Omega_n|\le|B_2|$ for all $n$. Moreover, the monotonicity of $\lambda_{p,q}$ with respect to set inclusion yields
$$\lambda_{p,q}(\Omega_n)\le \lambda_{p,q}(B_{1/2}).$$
Therefore, using the definition \eqref{prodotto}, we obtain
\[
\mathcal{F}_{p,q}(\Omega_n)
=
\lambda_{p,q}(\Omega_n)
\left(
\frac{|\Omega_n|^{1-\frac{1}{p}+\frac{1}{q}}}{P(\Omega_n)}
\right)^p
\le
\lambda_{p,q}(B_{1/2})
\left(
\frac{|B_2|^{1-\frac{1}{p}+\frac{1}{q}}}{P(\Omega_n)}
\right)^p.
\]
Since $P(\Omega_n)\to+\infty$ as $n\to\infty$, it follows that
\[
\mathcal{F}_{p,q}(\Omega_n)\to 0.
\]

This proves \eqref{infsimply}.
\end{proof}

The proof of the next proposition  heavily relies on the validity of the following Hardy inequality 
\begin{equation}\label{eq:hardy}
\int_A \frac{|u|^2}{d_A^2}\,dx\leq c  \int_A |\nabla u|^2\,dx\,,  \qquad \forall u\in C_0^\infty(A),
\end{equation}
that holds, with a positive constant $c$, on every planar simply connected open set $A$. The  optimal constant 
\[
\mathfrak h_2(A):=
\inf_{u\in C_0^\infty(A)}
\left\{
\int_A |\nabla u|^2\,dx
:
\int_A \frac{|u|^2}{d_A^2}\,dx =1
\right\}
\]
satisfies the uniform lower bound, due to  Ancona \cite[pag 208]{Ancona}, 
\begin{equation}\label{hardy}
\mathfrak h_2(A)\ge \frac1{16}.
\end{equation}

\begin{prop}\label{prop:supremum}
Let $p=2$ and $1\le q<2$. Then
\begin{equation}\label{supsimply}
\sup \big\{\mathcal{F}_{2,q}(\Omega): \Omega\subset\mathbb{R}^N \text{ simply connected open set}\big\}=+\infty.
\end{equation}
\end{prop}

\begin{proof}
We construct a sequence of simply connected open sets $\{\Omega_n\}_{n\in\mathbb N}$ such that
\[
\mathcal{F}_{2,q}(\Omega_n)\to+\infty
\qquad\mbox{as } n\to\infty.
\]

For every $n\in\mathbb N$, let
\[
A_n=(0,1)^2\setminus \left(\bigcup_{i=1}^{n-1}\left\{ \frac{i}{n} \right\} \times \left(0, 1-\frac{1}{n}\right)\right),
\]
and define
\[
\Omega_n=A_n\times (0,1)^{N-2}.
\]
The set $\Omega_n$ is open and simply connected. Moreover, since the slits removed from $A_{n}$ are one-dimensional, they do not affect the Lebesgue measure of $(0,1)^2$, hence
\[
\mathcal L^2(A_n)=1
\qquad\mbox{and therefore}\qquad
\mathcal L^N(\Omega_n)=\mathcal L^2(A_n)\,\mathcal L^{N-2}\big((0,1)^{N-2}\big)=1
\]
where  $\mathcal L^k(B)$ stands for the $k$-dimensional Lebesgue measure of (measurable) set $B\subseteq \mathbb R^k$. We first show that
\begin{equation}\label{lambda_lower_new}
\lambda_{2,q}(\Omega_n)\ge \lambda_{2,q}(A_n).
\end{equation}
To this end, let $u\in C_0^\infty(\Omega_n)$. For every $y\in(0,1)^{N-2}$, the function $x\mapsto u(x,y)$ belongs to $C_0^\infty(A_n)$.  
Hence, since  using Fubini's theorem and the definition of $\lambda_{2,q}(A_n)$, we obtain
\[
\begin{split}
\int_{\Omega_n} |\nabla u|^2\,dz
&\ge \int_{(0,1)^{N-2}}\int_{A_n} |\nabla_x u(x,y)|^2\,dx\,dy \\
&\ge \lambda_{2,q}(A_n)\int_{(0,1)^{N-2}}
\|u(\cdot,y)\|_{L^q(A_n)}^2\,dy.
\end{split}
\]
Since $q<2$ and $\mathcal L^{N-2}\big((0,1)^{N-2}\big)=1$, we have
\[
\int_{(0,1)^{N-2}}
\|u(\cdot,y)\|_{L^q(A_n)}^2\,dy
\ge
\left(\int_{(0,1)^{N-2}}
\|u(\cdot,y)\|_{L^q(A_n)}^q\,dy\right)^{2/q}
=
\|u\|_{L^q(\Omega_n)}^2.
\]
Hence
\[
\int_{\Omega_n} |\nabla u|^2\,dz
\ge
\lambda_{2,q}(A_n)\|u\|_{L^q(\Omega_n)}^2.
\]
Taking the infimum over all $u\in C_0^\infty(\Omega_n)$ such that $\|u\|_{L^q(\Omega_n)}=1$, we obtain \eqref{lambda_lower_new}. We now estimate $\lambda_{2,q}(A_n)$ from below. We argue as  in \cite[Subsection 1.2]{PZ} (see also \cite{vBe}): 
by using Hardy's inequality \eqref{eq:hardy} on the planar simply connected open set $A_n$ and then  applying Hölder's inequality, we get that, for every $n\in\mathbb N$ and for every  $v \in C_0^\infty(A_n)$, it holds 
\[
\int_{A_n} |v|^q \, dx \le \left( \int_{A_n} \frac{|v|^2}{d_{A_n}^{\,2} } \, dx \right)^\frac{q}{2}\, \|d_{A_n}\|_{L^{\frac{2\,q}{2-q}}(A_n)}^{{ q}}
\le (\mathfrak{h}_2(A_n))^{-\,\frac{q}{2}}   \left(\int_{A_n} |\nabla v|^2 \, dx \right)^{\frac{q}{2}}\,\|d_{A_n}\|_{L^{\frac{2\,q}{2-q}}(A_n)}^{{ q}}.
\]
Using the uniform estimate \eqref{hardy} and the inequality
\[
\|d_{A_n}\|_{L^{\frac{2q}{2-q}}(A_n)}
\le r_n\, \mathcal L^2(A_n)^{\frac{2-q}{2q}}
= r_n,
\]
where $r_n:=\|d_{A_n}\|_{L^\infty(A_n)}$,  we obtain that, for every $n\in \mathbb N$, 
\begin{equation}\label{eq:hardy2_new}
\lambda_{2,q}(A_n)\ge
\frac{\mathfrak h_2(A_n)}{\|d_{A_n}\|_{L^{\frac{2q}{2-q}}(A_n)}^2}
\ge
\frac{1}{16\,\|d_{A_n}\|_{L^{\frac{2q}{2-q}}(A_n)}^2} \ge \frac{1}{16\, r_n^2}.
\end{equation}

Finally, thanks to the fact that $|A_n\triangle (0,1)^2|=0$, we have that, for every $n\in \mathbb N$,  $A_n$ has finite perimeter equal to that of the unit square (see   \cite[Proposition 3.38]{AFP}), namely
\[
P(A_n)=P\big((0,1)^2\big)=4.
\]
Then we have
\[
P(\Omega_n)=P(A_n)\,\mathcal{L}^{N-2}\big((0,1)^{N-2}\big)+2\,\mathcal{L}^{2}(A_n)
=4\,\mathcal{L}^{N-2}\big((0,1)^{N-2}\big)+2\,\mathcal{L}^{2}(A_n)=6.
\]

Combining \eqref{lambda_lower_new} and \eqref{eq:hardy2_new}  with the above fact, we deduce
\[
\mathcal{F}_{2,q}(\Omega_n)
=
\lambda_{2,q}(\Omega_n)
\left(
\frac{\big( \mathcal L^N(\Omega_n) \big)^{1-\frac12+\frac1q}}{P(\Omega_n)}
\right)^2
\ge
c\, r_n^{-2},
\]
for some constant $c>0$ independent of $n$. Since,  by construction of the set $A_n$, we have that  $r_n\to 0$  as $n\to\infty$, it follows that
\[\mathcal{F}_{2,q}(\Omega_n)\to+\infty, \qquad \hbox{ as } n\to \infty \]
which proves \eqref{supsimply}.
\end{proof}


\begin{thebibliography}{100}

\bibitem{AFP} L.~Ambrosio, N.~Fusco, D.~Pallara, {\it Functions of Bounded
  Variation and Free Discontinuity Problems}, Oxford Mathematical Monographs,
  Oxford University Press, Oxford, 2000.

\bibitem{AFI} G.~Anello, F.~Faraci, A.~Iannizzotto, {\it On a problem of Huang concerning best constants in Sobolev embeddings}, {\it  Ann. Mat. Pura Appl. }~(4), {\bf 194} (2015), 767--779.


\bibitem{Ancona} A.~Ancona, On strong barriers and an inequality of Hardy for
  domains in $\mathbb{R}^N$, {\it J. London Math. Soc.}~(2) {\bf 34} (1986),
  274--290.
\newblock \href{https://doi.org/10.1112/jlms/s2-34.2.274}{https://doi.org/10.1112/jlms/s2-34.2.274}

\bibitem{AK} D.~H.~Armitage, \"U.~Kuran, The convexity of a domain and the
  superharmonicity of the signed distance function, {\it Proc. Amer. Math.
  Soc.} {\bf 93} (1985), 598--600.
\newblock \href{https://doi.org/10.1090/S0002-9939-1985-0776186-8}{https://doi.org/10.1090/S0002-9939-1985-0776186-8}

\bibitem{vBe} M.~van~den~Berg, Estimates for the torsion function and Sobolev
  constants, {\it Potential Anal.} {\bf 36} (2012), 607--616.
\newblock \href{https://doi.org/10.1007/s11118-011-9246-9}{https://doi.org/10.1007/s11118-011-9246-9}

\bibitem{bor73} C.~Borell, Integral inequalities for generalized concave or
  convex functions, {\it J. Math. Anal. Appl.} {\bf 43} (1973), 419--440.
\newblock \href{https://doi.org/10.1016/0022-247X(73)90083-8}{https://doi.org/10.1016/0022-247X(73)90083-8}

\bibitem{gar98} R.~J.~Gardner, G.~Zhang, Affine inequalities and radial mean
  bodies, {\it Amer. J. Math.} {\bf 120} (1998), no.~3, 505--528.
\newblock \href{https://doi.org/10.1353/ajm.1998.0021}{https://doi.org/10.1353/ajm.1998.0021}

\bibitem{Bra1} L.~Brasco, On principal frequencies and isoperimetric ratios
  in convex sets, {\it Ann. Fac. Sci. Toulouse Math.}~(6) {\bf 29} (2020),
  no.~4, 977--1005.
\newblock \href{https://doi.org/10.5802/afst.1653}{https://doi.org/10.5802/afst.1653}

\bibitem{Bra2} L.~Brasco, On principal frequencies and inradius in convex
  sets, {\it Bruno Pini Math. Anal. Semin.} {\bf 9}, Univ.~Bologna, Alma
  Mater Stud., Bologna, (2018), 78--101.
\newblock \href{https://doi.org/10.6092/issn.2240-2829/8945}{https://doi.org/10.6092/issn.2240-2829/8945}

\bibitem{BM} L.~Brasco, D.~Mazzoleni, On principal frequencies, volume and
  inradius in convex sets, {\it NoDEA Nonlinear Differ. Equ. Appl.} {\bf 27}
  (2020), no.~12, Paper No.~12, 27~pp.
\newblock \href{https://doi.org/10.1007/s00030-019-0614-2}{https://doi.org/10.1007/s00030-019-0614-2}

\bibitem{bradua} L.~Brasco, Convex duality for principal frequencies, {\it
  Math. Eng.} {\bf 4} (2022), no.~4, Paper No.~032, 28~pp.
\newblock \href{https://doi.org/10.3934/mine.2022032}{https://doi.org/10.3934/mine.2022032}

\bibitem{BPZ1} L.~Brasco, F.~Prinari, A.~C.~Zagati, A comparison principle
  for the Lane--Emden equation and applications to geometric estimates, {\it
  Nonlinear Anal.} {\bf 220} (2022), Paper No.~112847.
\newblock \href{https://doi.org/10.1016/j.na.2022.112847}{https://doi.org/10.1016/j.na.2022.112847}


\bibitem{BPZ2} L.~Brasco, F.~Prinari, A.~C.~Zagati,  Sobolev embeddings and distance functions, {\it Adv. Calc. Var.}, {\bf 17} (2024),
1365--1398

\bibitem{BBPconvex} L.~Briani, G.~Buttazzo, F.~Prinari, Some inequalities
  involving perimeter and torsional rigidity, {\it Appl. Math. Optim.} {\bf
  84} (2021), 2727--2741.
\newblock \href{https://doi.org/10.1007/s00245-020-09727-7}{https://doi.org/10.1007/s00245-020-09727-7}

\bibitem{BBP} L.~Briani, G.~Buttazzo, F.~Prinari, A shape optimization
  problem on planar sets with prescribed topology, {\it J. Optim. Theory
  Appl.} {\bf 193} (2022), 760--784.
\newblock \href{https://doi.org/10.1007/s10957-021-01870-7}{https://doi.org/10.1007/s10957-021-01870-7}

\bibitem{BBPtorsion} L.~Briani, G.~Buttazzo, F.~Prinari, Inequalities between
  torsional rigidity and principal eigenvalue of the $p$-Laplacian, {\it Calc.
  Var. Partial Differ. Equ.} {\bf 61} (2022), no.~2, Paper No.~78.
\newblock \href{https://doi.org/10.1007/s00526-021-02129-9}{https://doi.org/10.1007/s00526-021-02129-9}

\bibitem{BB} D.~Bucur, G.~Buttazzo, {\it Variational Methods in Shape
  Optimization Problems}, Progress in Nonlinear Differential Equations and
  their Applications {\bf 65}, Birkh\"auser Boston, Inc., Boston, MA, 2005.

\bibitem{DPG} F.~Della~Pietra, N.~Gavitone, Sharp bounds for the first
  eigenvalue and the torsional rigidity related to some anisotropic operators,
  {\it Math. Nachr.} {\bf 287} (2014), no.~2--3, 194--209.
\newblock \href{https://doi.org/10.1002/mana.201200296}{https://doi.org/10.1002/mana.201200296}

\bibitem{DM} P.~Dr\'abek, R.~Man\'asevich, On the closed solution to some
  nonhomogeneous eigenvalue problems with $p$-Laplacian, {\it Differ. Integral
  Equ.} {\bf 12} (1999), 773--788.
\newblock \href{https://doi.org/10.57262/die/1367241475}{https://doi.org/10.57262/die/1367241475}

\bibitem{Er} G.~Ercole, Absolute continuity of the best Sobolev constant of a bounded domain, J. Math. Anal. Appl., {\bf 404} (2013), 420--428.


\bibitem{FGL} I.~Fragal\`a, F.~Gazzola, J.~Lamboley, Sharp bounds for the
  $p$-torsion of convex planar domains, in: {\it Geometric Properties for
  Parabolic and Elliptic PDE's}, Springer INdAM Ser. {\bf 2}, Springer, Milan,
  (2013), 97--115.
\newblock \href{https://doi.org/10.1007/978-88-470-2841-8_7}{https://doi.org/10.1007/978-88-470-2841-8\_7}

\bibitem{GVB} M.~van~den~Berg, N.~Gavitone, On functionals involving the
  $p$-capacity and the $q$-torsional rigidity, {\it Calc. Var. Partial Differ.
  Equ.} {\bf 64} (2025), no.~8, Paper No.~245, 21~pp.
\newblock \href{https://doi.org/10.1007/s00526-025-03081-8}{https://doi.org/10.1007/s00526-025-03081-8}

\bibitem{lar} S.~Larson, A bound for the perimeter of inner parallel bodies,
  {\it J. Funct. Anal.} {\bf 271} (2016), 610--619.
\newblock \href{https://doi.org/10.1016/j.jfa.2016.02.022}{https://doi.org/10.1016/j.jfa.2016.02.022}

\bibitem{Makai} E.~Makai, On the principal frequency of a membrane and the
  torsional rigidity of a beam, in: {\it Studies in Mathematical Analysis and
  Related Topics}, Stanford Univ. Press, Stanford, (1962), 227--231.

\bibitem{Po} G.~P\'olya, Two more inequalities between physical and
  geometrical quantities, {\it J. Indian Math. Soc.} {\bf 24} (1960),
  413--419.

\bibitem{PZ} F.~Prinari, A.~C.~Zagati, On the sharp Makai inequality, {\it J.
  Convex Anal.} {\bf 31} (2024), no.~2, 709--732.

\bibitem{S} R.~Schneider, {\it Convex Bodies: The Brunn--Minkowski Theory}
  (2nd expanded edition), Encyclopedia Math. Appl. {\bf 151}, Cambridge Univ.
  Press, 2014.

\bibitem{JoSt} I.~Jo\'o, L.~Stach\'o, Generalization of an inequality of
  G.~P\'olya concerning the eigenfrequencies of vibrating bodies, {\it Publ.
  Inst. Math. (Beograd)} {\bf 31} (1982), 65--72.

\end{thebibliography}
\end{document}